\documentclass[a4wide,12pt]{amsart}
\usepackage{amssymb,latexsym}
\usepackage[all]{xy}
\usepackage{amscd}
\usepackage{epsfig}
\usepackage[usenames,dvipsnames]{pstricks}
\usepackage{vmargin}

\begin{document}
% One author
\title[Triplets of pure complexes]{Triplets of pure free squarefree complexes}
\author{Gunnar Fl{\o}ystad}
\address{Matematisk Institutt\\
         Johs. Brunsgt. 12\\ 
        5008 Bergen} 
\email{gunnar@mi.uib.no}

%\urladdr{http://webaddress}
%\thanks{thanks} 
% End one author

% Two authors
%\title[shorttitle]{titleline1\\ 
%                   titleline2}
%\author{Gunnar Fl{\o}ystad}
%\address{Matematisk Institutt\\
%         Johs. Brunsgt. 12\\
%         5008 Bergen}
%\email{gunnar@mi.uib.no}
%\urladdr{http://firstauthorwebaddress}
%\thanks{The firstauthor  ... thanks} 
%\author{secondauthor}
%\address{address2line1\\
%         address2line2\\
%         address2line3}
%\email{secondauthor@email address}
%\urladdr{http://secondauthorwebaddress}
%\thanks{The second author ... thanks}
% End two authors

\keywords{pure resolution, pure complex,
squarefree module, Betti numbers, homology, Alexander duality}
\subjclass[2010]{Primary: 13D02; Secondary: }
\date{\today}

% Option 5
% Theorems, corollaries, lemmas, and propositions, in the most 
% emphatic (plain) style; all are numbered separately.
% There is a Main Theorem in the most emphatic (plain) 
% style, unnumbered. There are definitions, in the less emphatic
% (definition) style. There are notations, in the least emphatic
%(remark) style, unnumbered.

\theoremstyle{plain}
\newtheorem{theorem}{Theorem}[section]
\newtheorem{corollary}[theorem]{Corollary}
\newtheorem*{main}{Main Theorem}
\newtheorem{lemma}[theorem]{Lemma}
\newtheorem{proposition}[theorem]{Proposition}
\newtheorem{conjecture}[theorem]{Conjecture}
\newtheorem*{theoremA}{Theorem}
\newtheorem*{theoremB}{Theorem}

\theoremstyle{definition}
\newtheorem{definition}[theorem]{Definition}
\newtheorem{question}[theorem]{Question}

\theoremstyle{remark}
\newtheorem{notation}[theorem]{Notation}
\newtheorem{remark}[theorem]{Remark}
\newtheorem{example}[theorem]{Example}
\newtheorem{claim}{Claim}

%Mine egne.

\newcommand{\psp}[1]{{{\bf P}^{#1}}}
\newcommand{\psr}[1]{{\bf P}(#1)}
\newcommand{\op}{{\mathcal O}}
\newcommand{\opw}{\op_{\psr{W}}}
\newcommand{\go}{\op}
\newcommand{\ini}[1]{\text{in}(#1)}
\newcommand{\gin}[1]{\text{gin}(#1)}
\newcommand{\kr}{{\Bbbk}}
\newcommand{\kk}{{\Bbbk}}
\newcommand{\pd}{\partial}
\newcommand{\vardel}{\partial}
\renewcommand{\tt}{{\bf t}}

%Kategorier

\newcommand{\coh}{{{\text{{\rm coh}}}}}

%Modulkategorier

\newcommand{\modv}[1]{{#1}\text{-{mod}}}
\newcommand{\modstab}[1]{{#1}-\underline{\text{mod}}}

\newcommand{\sut}{{}^{\tau}}
\newcommand{\sumit}{{}^{-\tau}}
\newcommand{\til}{\thicksim}

\newcommand{\totp}{\text{Tot}^{\prod}}
\newcommand{\dsum}{\bigoplus}
\newcommand{\dprod}{\prod}
\newcommand{\lsum}{\oplus}
\newcommand{\lprod}{\Pi}
\newcommand{\sq}{\text{sq}\!-\!S}
\newcommand{\fsq}{\text{fsq}\!-\!S}
\newcommand{\smod}{{}^*\!\text{mod}\!-\!S}

%Funktorer
\newcommand{\sqfu}{{\mathcal Sq}}

% Algebraer
\newcommand{\La}{{\Lambda}}
\newcommand{\lam}{{\lambda}}
\newcommand{\GL}{{GL}}

\newcommand{\sirstj}{\circledast}

% Knipper
\newcommand{\she}{\EuScript{S}\text{h}}
\newcommand{\cm}{\EuScript{CM}}
\newcommand{\cmd}{\EuScript{CM}^\dagger}
\newcommand{\cmri}{\EuScript{CM}^\circ}
\newcommand{\cler}{\EuScript{CL}}
\newcommand{\clerd}{\EuScript{CL}^\dagger}
\newcommand{\clerri}{\EuScript{CL}^\circ}
\newcommand{\gor}{\EuScript{G}}
\newcommand{\gF}{\mathcal{F}}
\newcommand{\gG}{\mathcal{G}}
\newcommand{\gM}{\mathcal{M}}
\newcommand{\gE}{\mathcal{E}}
\newcommand{\gD}{\mathcal{D}}
\newcommand{\gI}{\mathcal{I}}
\newcommand{\gP}{\mathcal{P}}
\newcommand{\gK}{\mathcal{K}}
\newcommand{\gL}{\mathcal{L}}
\newcommand{\gS}{\mathcal{S}}
\newcommand{\gC}{\mathcal{C}}
\newcommand{\gO}{\mathcal{O}}
\newcommand{\gJ}{\mathcal{J}}
\newcommand{\gU}{\mathcal{U}}
\newcommand{\gX}{\mathcal{X}}
\newcommand{\mm}{\mathfrak{m}}

\newcommand{\dlim} {\varinjlim}
\newcommand{\ilim} {\varprojlim}

%Kategorier
\newcommand{\CM}{\text{CM}}
\newcommand{\Mon}{\text{Mon}}

%Kategorieer av komplekser

\newcommand{\Kom}{\text{Kom}}

% Begreper homologisk alebra

\newcommand{\EH}{{\mathbf H}}
\newcommand{\res}{\text{res}}
\newcommand{\Hom}{\text{Hom}}
\newcommand{\inhom}{{\underline{\text{Hom}}}}
\newcommand{\Ext}{\text{Ext}}
\newcommand{\Tor}{\text{Tor}}
\newcommand{\ghom}{\mathcal{H}om}
\newcommand{\gext}{\mathcal{E}xt}
\newcommand{\id}{\text{{id}}}
\newcommand{\im}{\text{im}\,}
\newcommand{\codim} {\text{codim}\,}
\newcommand{\resol}{\text{resol}\,}
\newcommand{\rank}{\text{rank}\,}
\newcommand{\lpd}{\text{lpd}\,}
\newcommand{\coker}{\text{coker}\,}
\newcommand{\supp}{\text{supp}\,}
\newcommand{\Ad}{A_\cdot}
\newcommand{\Bd}{B_\cdot}
\newcommand{\dt}{{\bullet}}
\newcommand{\Fd}{F^\dt}
\newcommand{\Gd}{G^\dt}
\newcommand{\Xd}{X^\dt}
\newcommand{\Xdd}{X_\dt}
\newcommand{\omdot}{\omega_S^\cdot}

%Avbildninger og andre symbolforkortelser

\newcommand{\sus}{\subseteq}
\newcommand{\sups}{\supseteq}
\newcommand{\pil}{\rightarrow}
\newcommand{\vpil}{\leftarrow}
\newcommand{\rpil}{\leftarrow}
\newcommand{\lpil}{\longrightarrow}
\newcommand{\inpil}{\hookrightarrow}
\newcommand{\pils}{\twoheadrightarrow}
\newcommand{\projpil}{\dashrightarrow}
\newcommand{\dotpil}{\dashrightarrow}
\newcommand{\adj}[2]{\overset{#1}{\underset{#2}{\rightleftarrows}}}
\newcommand{\mto}[1]{\stackrel{#1}\longrightarrow}
\newcommand{\vmto}[1]{\overset{\tiny{#1}}{\longleftarrow}}
\newcommand{\mtoelm}[1]{\stackrel{#1}\mapsto}

\newcommand{\eqv}{\Leftrightarrow}
\newcommand{\impl}{\Rightarrow}

\newcommand{\iso}{\cong}
\newcommand{\te}{\otimes}
\newcommand{\into}[1]{\hookrightarrow{#1}}
\newcommand{\ekv}{\Leftrightarrow}
\newcommand{\equi}{\simeq}
\newcommand{\isopil}{\overset{\cong}{\lpil}}
\newcommand{\equipil}{\overset{\equi}{\lpil}}
\newcommand{\ispil}{\isopil}
\newcommand{\vvi}{\langle}
\newcommand{\hvi}{\rangle}
\newcommand{\susneq}{\subsetneq}
\newcommand{\sgn}{\text{sign}}
\newcommand{\prikk}{\bullet}

%Notasjonsforkortelser

\newcommand{\xd}{\check{x}}
\newcommand{\ortog}{\bot}
\newcommand{\tL}{\tilde{L}}
\newcommand{\tM}{\tilde{M}}
\newcommand{\tH}{\tilde{H}}
\newcommand{\tvH}{\widetilde{H}}
\newcommand{\tvh}{\widetilde{h}}
\newcommand{\tV}{\tilde{V}}
\newcommand{\tS}{\tilde{S}}
\newcommand{\tT}{\tilde{T}}
\newcommand{\tR}{\tilde{R}}
\newcommand{\tf}{\tilde{f}}
\newcommand{\ts}{\tilde{s}}
\newcommand{\tp}{\tilde{p}}
\newcommand{\tr}{\tilde{r}}
\newcommand{\tfst}{\tilde{f}_*}
\newcommand{\empt}{\emptyset}
\newcommand{\bfa}{{\bf a}}
\newcommand{\bfb}{{\bf b}}
\newcommand{\bfd}{{\bf d}}
\newcommand{\bfe}{{\bf e}}
\newcommand{\bfp}{{\bf p}}
\newcommand{\bfc}{{\bf c}}
\newcommand{\bfl}{{\bf t}}
\newcommand{\la}{\lambda}
\newcommand{\bfen}{{\mathbf 1}}
\newcommand{\ep}{\epsilon}
\newcommand{\en}{{\mathbf 1}}
\newcommand{\tu}{s}
\newcommand{\carc}{\mbox{char.}}

\newcommand{\ome}{\omega_E}

\newcommand{\bevis}{{\bf Proof. }}
\newcommand{\demofin}{\qed \vskip 3.5mm}
\newcommand{\nyp}[1]{\noindent {\bf (#1)}}
\newcommand{\demo}{{\it Proof. }}
\newcommand{\demodone}{\demofin}
\newcommand{\parg}{{\vskip 2mm \addtocounter{theorem}{1}  
                   \noindent {\bf \thetheorem .} \hskip 1.5mm }}

\newcommand{\lcm}{{\text{lcm}}}
\newcommand{\sA}{{\mathbb A}}
\newcommand{\sD}{{\mathbb D}}
\newcommand{\sAD}{\sA \circ \sD}

% Simplisielle komplekser

\newcommand{\dl}{\Delta}
\newcommand{\cdel}{{C\Delta}}
\newcommand{\cdelp}{{C\Delta^{\prime}}}
\newcommand{\dlst}{\Delta^*}
\newcommand{\Sdl}{{\mathcal S}_{\dl}}
\newcommand{\lk}{\text{lk}}
\newcommand{\lkd}{\lk_\Delta}
\newcommand{\lkp}[2]{\lk_{#1} {#2}}
\newcommand{\del}{\Delta}
\newcommand{\delr}{\Delta_{-R}}
\newcommand{\dd}{{\dim \del}}

%Notasjoner
\renewcommand{\aa}{{\mathbf a}}
\newcommand{\bb}{{\mathbf b}}
\newcommand{\cc}{{\mathbf c}}
\newcommand{\xx}{{\mathbf x}}
\newcommand{\yy}{{\mathbf y}}
\newcommand{\zz}{{\mathbf z}}
\newcommand{\rr}{{\mathbf r}}
\newcommand{\bA}{{\mathbf A}}

\newcommand{\mv}{{\xx^{\aa_v}}}
\newcommand{\mF}{{\xx^{\aa_F}}}

\newcommand{\pnm}{{\mathbf P}^{n-1}}
\newcommand{\opnm}{{\go_{\pnm}}}
\newcommand{\ompnm}{\omega_{\pnm}}

\newcommand{\pn}{{\mathbf P}^n}
\newcommand{\hele}{{\mathbb Z}}
\newcommand{\nat}{{\mathbb N}}
\newcommand{\rasj}{{\mathbb Q}}

\newcommand{\st}{\hskip 0.5mm {}^{\rule{0.4pt}{1.5mm}}}              
\newcommand{\disk}{\scriptscriptstyle{\bullet}}
\newcommand{\sirk}{\circ}

\newcommand{\cF}{F_\dt}
\newcommand{\pol}{f}

\newcommand{\disc}{\circle*{5}}

\newcommand{\tri}{\Delta}
\newcommand{\hal}{\hat{\alpha}}
\newcommand{\hbe}{\hat{\beta}}
\newcommand{\hga}{\hat{\gamma}}
\newcommand{\ovpil}[1]{\overset{\pil}{#1}}
\newcommand{\pwov}{{\mathbb P}(\overset{\pil}{W})}
\newcommand{\trs}[1]{\overline{#1}}
\newcommand{\Sym}{\mbox{Sym}}
\newcommand{\gQ}{\mathcal Q}
\newcommand{\gR}{\mathcal R}
\newcommand{\Spec}{\mbox{Spec}}
\newcommand{\rk}{\mbox{rk}}

\def\CC{{\mathbb C}}
\def\GG{{\mathbb G}}
\def\ZZ{{\mathbb Z}}
\def\NN{{\mathbb N}}
\def\RR{{\mathbb R}}
\def\OO{{\mathbb O}}
\def\QQ{{\mathbb Q}}
\def\VV{{\mathbb V}}
\def\PP{{\mathbb P}}
\def\EE{{\mathbb E}}
\def\FF{{\mathbb F}}
\def\AA{{\mathbb A}}
\def\DD{{\mathbb D}}
\def\AD{{\AA \circ \DD}}

\begin{abstract}
On the category of bounded complexes of finitely generated
free squarefree modules over 
the polynomial ring $S$, there
is the standard duality functor $\DD = \Hom_S(-, \omega_S)$ and
the Alexander duality functor $\AA$. The composition $\AD$ is an endofunctor
on this category, of order three
up to translation. We consider complexes $F^\dt$ of free squarefree modules
such that both $F^\dt, \AD(F^\dt)$ and $(\AD)^2(F^\dt)$ are pure, when
considered as singly graded complexes. We conjecture i) the existence
of such triplets of complexes for given triplets of degree sequences, 
and ii) the uniqueness of their Betti numbers, up to scalar multiple.
We show that this uniqueness follows from the existence, and
we construct such triplets if two of them are linear.
\end{abstract}
\maketitle

\section*{Introduction}

\noindent{\bf Pure free resolutions} are
free resolutions over the polynomial ring $S$ of the form
\[ S(-d_0)^{\beta_0} \vpil S(-d_1)^{\beta_1} \vpil \cdots
\vpil S(-d_r)^{\beta_r}.\]
Their Betti diagrams have proven to be of fundamental importance
in the study of Betti diagrams of graded modules over the polynomial 
ring. Their significance were put to light by the Boij-S\"oderberg
conjectures, \cite{BS}. The existence of pure resolutions
were first proven by D.Eisenbud, the author, and J.Weyman in 
\cite{EFW} in characteristic zero, and by D.Eisenbud and F.-O.Schreyer
in all characteristics, \cite{ES}. Later the methods of 
\cite{ES} were made more explicit and put into a larger framework, 
called tensor complexes, by C.Berkesch et.al. \cite{BEKS2}.

The Boij-S\"oderberg conjectures, settled in full generality in \cite{ES},
concerns the stability theory of {\it Betti diagrams} of graded
modules, i.e. it describes such diagrams up to multiplication by a
positive rational number, or alternatively the positive rational
cone generated by such diagrams. The Betti diagrams of pure resolutions are exactly the
extremal rays in this cone. Two introductory papers on this theory
are \cite{FlIntro} and \cite{ESBeSy}.

\medskip
\noindent{\bf Homological invariants.} The Betti diagram
is however only part of the story when it comes to homological invariants
of graded modules. A complex $F^\dt$ of free modules over the 
polynomial ring $S$, for instance a free resolution, comes with
three sets of numerical homological invariants:
\begin{itemize}
\item $B$: The { graded
Betti numbers} $\{\beta_{ij}\}$,
\item $H$: The Hilbert functions of the { homology modules} $H^i(F^\dt)$,
\item $C$: The Hilbert functions of the { cohomology modules}. These
modules are the 
homology modules of the dual complex $\Hom_S(F^\dt, \omega_S)$, where
$\omega_S$ is the canonical module.
\end{itemize}
It is then natural to approach the stability theory 
of the triplet data set $(B,H,C)$: Up to rational multiple, what
triplets of such can occur? 
The recent article \cite{EE-Cat} has partial results in this direction.
It describes the Betti diagrams of complexes $F^\dt$ with 
specified nondecreasing codimensions of the homology modules. 
We do not here investigate the question above directly,
but we believe the following will be of relevance.

\medskip
\noindent{\bf Squarefree modules.} 
The notion of pure resolution or pure complex, has a very natural
extension into {\it triplets of pure complexes}, in the setting of 
squarefree modules over the polynomial ring. 
Squarefree modules are $\NN^n$-graded modules over the polynomial ring
$S = \kk[x_1, \ldots, x_n]$ and form a module category including
squarefree monomial ideals, and Stanley-Reisner rings.
Both the category of singly graded $S$-modules as well as 
squarefree $S$-modules, have the standard duality functor 
$\DD = \Hom_S(-, \omega_S)$. However for squarefree modules
there is also another duality functor, Alexander duality $\AA$.
The composition $\AD$ becomes an endofunctor
on the category of bounded complexes of finitely generated free
squarefree $S$-modules. (This is in fact the Auslander-Reiten
translate on the derived category of complexes of squarefree modules, see
\cite{BrFl}.) There are two amazing facts concerning this endofunctor.
\begin{itemize}
\item The third iterate $(\AD)^3$ is isomorphic to the $n$'th iterate
of the translation functor on complexes, a result of K.Yanagawa, \cite{Ya}.
\item The composition functor cyclically rotates the homological
invariants: If $F^\dt$ has homological invariants $(B,H,C)$ then
\begin{itemize}
\item $\AD(F^\dt)$ has homological invariants $(H,C,B)$, and
\item $(\AD)^2(F^\dt)$ has homological invariants $(C,B,H)$. 
\end{itemize}
This is also implicit in \cite{Ya}.
\end{itemize}
Thus the various homology modules of $F^\dt$ are transferred
to the various linear strands of $\AD(F^\dt)$ and the cohomology
modules of $F^\dt$ are transferred to the linear strands of 
$(\AD)^2(F^\dt)$. 

\medskip
\noindent{\bf The main idea} of this paper is to
consider complexes $F^\dt$ of free squarefree modules such that
(when considered as singly graded modules)
\begin{itemize}
\item $F^\dt$ is pure,
\item $\AD(F^\dt)$ is pure, 
\item $(\AD)^2(F^\dt)$ is pure.
\end{itemize}
We call this a {\it triplet of pure complexes}. 
That
$F^\dt$ is a pure resolution of a Cohen-Macaulay squarefree module, 
the classical case, corresponds
to
\begin{itemize}
\item $F^\dt$ is pure,
\item $\AD(F^\dt)$ is linear, 
\item $(\AD)^2(F^\dt)$ is linear.
\end{itemize}

\medskip
\noindent{\bf Construction of triplets.} Squarefree complexes are
$\ZZ^n$-graded, or equivalently they are equivariant for the action
of the diagonal matrices of $GL(n)$. That pure resolutions 
come with various
group actions is the rule in the various constructions we have,
\cite{EFW}, \cite{BEKS2}. S.Sam and J.Weyman pursue this \cite{SW}
in the context of other linear algebraic groups. 
However being squarefree is something more than being $\ZZ^n$-graded.
In particular for a squarefree complex $F^\dt$ it may happen that
the only multidegree $\bb$ such that $F^\dt(-\bb)$ is squarefree,
is the zero degree.  
It is therefore a priori not clear, even in the classical case, 
how to construct such complexes $F^\dt$. 
As it turns out the tensor complexes of \cite{BEKS2} 
make the perfect 
input for a construction, see in particular Remark \ref{KonRemTensor}.
These tensor complexes are
over a large polynomial ring $S(V \te W_0^* \te \cdots \te W_{r+1}^*)$. 
Letting $V$ be the linear space $\langle x_1, \ldots, x_n \rangle$
and taking a general map
\[ V \te W_0^* \te \cdots \te W_{r+1}^* \pil V, \]
equivariant for the diagonal matrices
in $GL(n)$,
we may construct all cases of complexes $F^\dt$ corresponding to 
the classical case,
Theorem \ref{KonTeoremHoved}. 
The existence of triplets of pure complexes in full generality, we
state as Conjecture \ref{PurecxConj}. In
a subsequent paper, \cite{FlZip}, we transfer this to a conjecture
on the existence of certain complexes of coherent sheaves on 
projective spaces.

\medskip
\noindent{\bf Uniqueness of Betti numbers.} In the classical case
the singly graded Betti numbers of $F^\dt$ (and also of 
the linear complexes $\AD (F^\dt)$ and $(\AD)^2(F^\dt)$ are uniquely determined
up to scalar multiple, by the degree sequence of $F^\dt$.

It now turns out that for a triplet of pure complexes, given the
degree sequences of each of the three complexes, the Betti numbers
fulfill a number of homogeneous linear equations which is one less than 
the number of variables, i.e. the number of Betti numbers. 
We thus expect there to
be a unique solution up to common rational multiple. Under
the assumption that triplets of pure complexes exists (for all 
triplet of degree sequences fulfilling a simple necessary criterion),
we show that the Betti numbers are uniquely determined up to 
common rational multiple, Theorem \ref{NumTeoremAB} 

\medskip
\noindent{\bf Pure resolutions in the squarefree setting} have previously
been considered by W.Bruns and T.Hibi for Stanley-Reisner rings.
In \cite{BHSr} they describe all possible degree sequences 
$0 = d_0, d_1, \ldots$
for pure resolutions of Stanley-Reisner rings with $d_1 = 2$, 
and classify the simplicial complexes where this occurs.
When $d_1 = 3$ they give a thorough investigation of possible degree
sequences and the possible simplicial complexes, as
well as interesting examples when $d_1 \geq 4$. They also give a complete
classification of simplicial complexes where $d_1 = m$ and $d_2 = 2m-1$
for $m \geq 2$.  
In \cite{BHCM} they
classify Cohen-Macaulay posets where the Stanley-Reisner 
ring of the order complex
has pure resolution. 
In \cite{FlEnrich} the author considers {\it Cohen-Macaulay designs} which
in the language of the present article correspond to
Cohen-Macaulay Stanley-Reisner rings
with pure resolution and exactly three linear strands (so the 
Stanley-Reisner ideal has exactly two linear strands).
Examples of such are cyclic polytopes and Alexander duals of Steiner
systems.

However from the perspective of the present article, 
approaches in those
directions are severely hampered by the fact that only for
few degree sequences, by simple numerical considerations, 
can one hope that the first Betti number
$\beta_0$ may be chosen to be $1$. 
For degree sequences where this
value may be achieved these articles also testify to
the difficulty in constructing pure resolutions of Stanley-Reisner rings.
Our construction avoids the restriction $\beta_0 = 1$, rather making
$\beta_0$ large.

\medskip
\noindent{\bf Organization of article.}
In Section \ref{SetSec} we give the setting of squarefree modules
and the functors $\AA$ and $\DD$. We show that they rotate the 
homological invariants of squarefree complexes.
In Section \ref{PurecxSec} we develop the basic theory of triplets
of pure complexes. We find a basic necessary condition, the balancing
condition, on the
triplet of degree sequences of such complexes. We conjecture
the existence of triplets of pure complexes 
for all balanced triplets of degree sequences,
and the uniqueness of their Betti numbers, up to common scalar multiple,
Conjecture \ref{PurecxConj}.
In Section \ref{NumSec} we show this uniqueness of Betti numbers, 
under the assumption that
triplets of pure complexes do exist. In Section \ref{KonSec}
we use the tensor complexes of \cite{BEKS2} to construct 
triplets of pure complexes $F^\dt, \AD(F^\dt)$ and $(\AD)^2(F^\dt)$
when the last two complexes are linear.

\section{Duality functors and rotation of homological invariants}
\label{SetSec}

In this section we recall the notion of a squarefree module over the
polynomial ring $S = \kk[x_1, \ldots, x_n]$, and the two duality
functors we may define on the category of complexes of such modules, 
standard duality $\sD$ and Alexander duality $\sA$. 

A striking result of K.Yanagawa \cite{Ya}, says that the composition 
$(\sA \circ \sD)^3$ is naturally equivalent to the $n$'th iterate of the
translation functor
on the derived category of squarefree modules. A complex of
squarefree modules comes with three sets of homological invariants, 
the multigraded homology and cohomology modules, and the multigraded
Betti spaces. We show that $\sA \circ \sD$ cyclically rotates these
invariants (which is a rather well known fact to experts).

\subsection{Squarefree modules and dualities}
\label{SetSubsecSqf}

Let $S$ be the polynomial ring $\kk[x_1, \ldots, x_n]$ where $\kk$ is a field. 
Let $\ep_i$ be the $i$'th coordinate vector in $\nat^n$.
An $\nat^n$-graded $S$-module is called {\it squarefree}, introduced
by K.Yanagawa in \cite{YaSqf}, if $M$ is finitely
generated and the multiplication map 
$M_{\bb} \mto{\cdot x_i} M_{\bb + \ep_i}$ is an isomorphism of vector spaces
whenever the $i$'th coordinate $b_i \geq 1$. 
We denote the category of finitely generated squarefree $S$-modules by $\sq$.

There is a one-one correspondence between subsets 
$R \sus [n] = \{1,2,\ldots,n\}$ and
multidegrees $\rr$ in $\{0,1\}^n$, by letting $R$ be the set of coordinates
of $\rr$ equal to $1$. By abuse of notation we shall often write $R$ when
strictly speaking we mean $\rr$. For instance the degree $\rr$ part
of $M$, which is, $M_{\rr}$ may be written $M_R$.
Also if $R$ is a set we shall if no confusion arises, denote its
cardinality by the smaller case letter $r$.
%It will be convenient to identify subsets $R$ of $[n] = \{1,2, \ldots, n\}$
%with its characteristic vector $\rr \in \{0,1\}^n$. We see that a 
%squarefree module is determined by its module structure in degrees $M_R$
%where $R \sus [n]$. 
We also denote $(1,1,\ldots,1)$ as $\en$. Take note that
a squarefree module is completely determined, up to isomorphism,
by the graded pieces $M_R$ and the multiplication maps between them
\[ M_R \mto{x_v} M_{R \cup \{ v\}} \]
where $v \not \in R$. 

If $M$ is a squarefree module and $0 \leq d \leq n$, its {\it squarefree
part of degree $d$} is 
\[ \bigoplus_{|R| = d} M_R. \]
Note that taking squarefree parts is an exact functor from squarefree
modules to vector spaces. In particular note that the squarefree part
of $S(-\bb)$ in degree $d$ has dimension $\binom{n-|\bb|}{d - |\bb|}
= \binom{n - |\bb|}{n-d}$. 

For a squarefree module $M$ there is a notion of Alexander
dual module $\bA(M)$, defined by T.R{\"o}mer \cite{RoAlex}
and E.Miller \cite{MiAlex}. For $R$ a subset of $[n]$, let $R^c$ be its complement. 
Then $\bA(M)_R$ is the dual $\Hom_\kk(M_{R^c},k)$.
If $v$ is not in $R$ the multiplication
\[ \bA(M)_R \mto{\cdot x_v} \bA(M)_{R \cup \{v\}}  \]
is the dual of the multiplication
\[ M_{(R\cup \{v\})^c} \mto{\cdot x_v} M_{R^c}. \]
By obvious extension this defines $\bA(M)_\bb$ for all $\bb$ in $\nat^n$
and all multiplications.

\begin{example}
If $S = \kk[x_1, x_2, x_3, x_4]$ then the Alexander dual of $S(-(1,0,1,1))$
is $S/(x_1, x_3, x_4)$. The Alexander dual of $S(- \en)$ is the simple
quotient module $k$. 
\end{example}

For a multidegree $\bb$ in $\hele^n$, the free $S$-module $S(-\bb)$
is a squarefree module if and only if  $\bb \in \{0,1\}^n$, i.e.
all coordinates of $\bb$ are $0$ or $1$. 
Direct sums of such modules are the {\it free squarefree $S$-modules}. 
Denote by $\fsq$ the category of finitely generated such modules.

Let $C^b(\sq)$ and $C^b(\fsq)$ be the categories of bounded complexes
of finitely generated squarefree, resp. free squarefree modules.
There is a natural duality 
\[ \DD: C^b(\fsq)  \lpil C^b(\fsq) \]
defined by 
\[ \DD(F^\dt) = \Hom_S(F^\dt, S(-\en)), \] so in 
particular $\DD(S(-\bb)) = S(\bb - \en)$. 
We would also like to define Alexander duality on the category
$C^b(\fsq)$. 
However there is a slight problem in that Alexander duality
as defined above does not take free modules to free modules.

To remedy this, any bounded complex of squarefree modules $X^\dt$ has a minimal
resolution $F^\dt \pil X^\dt$ by free squarefree modules.
This defines a functor $\res: C^b(\sq) \pil C^b(\fsq)$. 
(There is of course also a natural inclusion $\iota: C^b(\fsq) \pil C^b(\sq)$.) 
We now define Alexander duality
\[ \AA: C^b(\fsq) \pil C^b(\fsq) \]
by letting $\AA$ be the composition $\res \circ \bA$ where
$\bA$ is the Alexander duality defined above.

\begin{example}
Continuing the example above, a free resolution of $S/(x_1, x_2, x_4)$
is 
\begin{equation*}
\underset{\scriptsize{\begin{matrix} (0,0,0,0) \end{matrix}}} 
{S} \longleftarrow 
\underset{\scriptsize{\begin{matrix} (1,0,0,0)\\ (0,0,1,0) \\ (0,0,0,1) 
\end{matrix}}} {S^3} \longleftarrow
\underset{\scriptsize{\begin{matrix} (1,0,1,0) \\ (1,0,0,1) \\ (0,0,1,1)
 \end{matrix}}} {S^3} \longleftarrow
\underset{\scriptsize{\begin{matrix} (1,0,1,1) \end{matrix}}} 
{S} 
\end{equation*}
where we have written below the multidegrees of the generators.
Then the Alexander dual $\AA (S(-(1,0,1,1))$ is the above resolution.
\end{example}

By composing with the resolution 
%$\res: C^b(\sq) \pil C^b(\fsq)$
we may also consider $\AA$ and $\DD$ as functors on $C^b(\sq)$
\[ C^b(\sq) \mto{\res} C^b(\fsq) \mto{\AA, \DD} C^b(\fsq). \]

For a complex $X^\dt$ let $X^\dt[-p]$ be its $p$'th cohomological translate,
i.e. $(X^\dt[-p])^q = X^{q-p}$. Yanagawa, \cite{Ya},
shows that $(\AA \circ \DD)^3$ is isomorphic to the $n$'th iterate $[n]$ 
of the translation functor.

\subsection{Homological invariants}

%Any complex $X^\dt$ in $K^b(\sq)$ has a free resolution
%$F^\dt \mto{\simeq} X^\dt$, meaning that $F^\dt$ is a complex of free
%squarefree modules and the map is a quasi-isomorphism.
%One may also assume that $F^\dt$ is {\it minimal}, 
%meaning that all differentials in $F^\dt \te_S k$ are zero.
%In this case $F^\dt$ is unique up to isomorphism.
%The complex $F^\dt$ is said to be linear if considered 
%as a $\hele$-graded complex, there is a fixed integer $a$ such that 
%each term may be written as 
%$F^i = S \te_\kk B^i_{-i-a}$ where $B^i_{-i-a}$ is a vector space
%of degree $-i-a$. 

The complex $X^\dt$ comes with three sets of squarefree homological 
invariants. First there is the homology 
\[ H^i_R(X^\dt) := (H^i(X^\dt))_R \]
where $i \in \hele$ and $R \sus [n]$. 
%Secondly there is the cohomology. It will be convenient to work with 
%the duals of the spaces that we usually think of as the cohomology spaces
%and call these the cohomology* spaces.
For a vector space $V$, denote by $V^*$ its dual $\Hom_\kk(V,\kk)$.
We define the cohomology as 
\[ C^i_R(X^\dt) := (H^{-i}(\sD (X^\dt))_{R^c})^*. \]
Note that by local duality, if $X^\dt$ is a module $M$, then
this relates to local cohomology by 
\[C^{i-n}_R(M) = H^i_{\frak m}(M)_{\rr - \en} \]
where $\rr$ is the $0,1$-vector with support $R$.
Thirdly a minimal free squarefree resolution $F^\dt$ of $X^\dt$ 
has terms which may be
written $F_i = \oplus_{R\sus[n]} S \te_\kk B^i_R$ and we define the Betti spaces 
to be 
\[ B^i_R(X^\dt) := (Tor_i^S(X^\dt,k)_R) = (B^i_R). \]

\medskip
Now a basic and very interesting fact is that the functors $\sA$ and
$\sD$ interchange the homology, cohomology, and Betti spaces.
First we consider $\DD$.

\begin{lemma}
The functor $\sD$ interchanges the homological invariants of $X^\dt$ as
follows.
\begin{itemize}
\item $B^i_R(\sD (X^\dt)) = B^{-i}_{R^c}(X^\dt)^*$.
\item $H^i_R(\sD (X^\dt)) = C^{-i}_{R^c}(X^\dt)^*$.
\item $C^i_R(\sD (X^\dt)) = H^{-i}_{R^c}(X^\dt)^*$.
\end{itemize}
\end{lemma}

\begin{proof}
This is clear.
\end{proof}

Before describing how the functor $\sA$ interchanges the homological
invariants, we recall a basic fact from \cite{Ya}.
For a square-free module $M$, one may define a complex $\gL(M)$ (see
\cite[p.9]{Ya} where it is denoted by $\gF(M)$) by 
\[\gL^i(M) = \bigoplus_{|R| = i} (M_R)^\circ \te_\kk S\] 
where $(M_R)^\circ$ is $M_R$ but considered to have multidegree
$R^c$. The 
differential is 
\[ m^\circ \te s \mapsto \sum_{j \not \in R} (-1)^{\alpha(j,R)} (x_jm)^\circ
\te x_js \]
where $\alpha(j,R)$ is the number of $i$ in $R$ such that $i < j$.

 For a minimal complex $F^\dt$ of free squarefree $S$-modules define its 
{\it $i$'th linear strand
$F^\dt_{\langle i \rangle}$} to have terms 
\[ F^{j}_{\langle i \rangle} = \bigoplus_{|R| = i-j} S \te_\kk B^j_R. \]
Since $F^\dt$ is minimal, the $i$'th linear strand is naturally
a complex.
The following is \cite[Thm. 3.8]{Ya}.

\begin{proposition} \label{HexProDA} The $i$'th linear strand of 
$\sA \circ \sD (X^\dt)$ is
\[ \gL(H^i(X^\dt))[n-i].\]
\end{proposition}

This gives the following.

\begin{lemma} The functor $\sA$ interchanges the homological invariants of 
$X^\dt$ as follows (denoting the cardinality of $R$ by $r$).
\begin{itemize}
\item[a.] $B^i_R(\sA (X^\dt)) = C^{-i-r}_{R}(X^\dt)^*$.
\item[b.] $H^i_R(\sA (X^\dt)) = H^{-i}_{R^c}(X^\dt)^*$.
\item[c.] $C^i_R(\sA (X^\dt)) = B^{-i-r}_{R}(X^\dt)^*$.
\end{itemize}
\end{lemma}

\begin{proof} Part b. is clear.
By the proposition above 
\[\gL(H^i(\sD (X^\dt)))[n-i] \iso \sA(X^\dt)_{\langle i \rangle}. \]
The first complex has terms which are direct sums over $R$ of 
\[ H^i(\DD(X^\dt)_R)^\circ \te_\kk S\]
where the generating space has internal degree $R^c$ and 
is in cohomological position $r+i-n = i - |R^c|$.
The generating space here is 
\[ (C^{-i}(X^\dt)_{R^c}^*)^\circ. \]
Hence this equals
\[ B^{i-|R^c|}_{R^c}(\AA(X^\dt)),\]
which is equivalent to part a.

 Part c. follows from a. by replacing 
$X^\dt$ by $\sA(X^\dt)$. 
\end{proof}

Putting these two lemmata together we get the following.

\begin{corollary}
The composition $\sA \circ \sD$ cyclically rotates the homological
invariants as follows.
\begin{itemize}
\item $B^i_R(\sA \circ \sD (X^\dt)) = H^{i+r}_{R^c}(X^\dt)$.
\item $H^i_R(\sA \circ \sD (X^\dt)) = C^{i}_{R}(X^\dt)$.
\item $C^i_R(\sA \circ \sD (X^\dt)) = B^{i+r}_{R^c}(X^\dt)$.
\end{itemize}
\end{corollary}

We may depict the rotation of homological invariants by the diagram

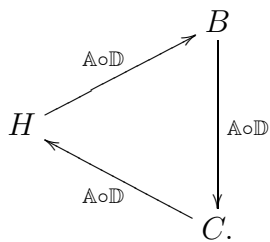
\begin{center}
\hskip 0mm
\xymatrix{  & & B  \ar[dd]^{\sA \circ \sD} \\
          H \ar[urr]^{\sA \circ \sD}  & &  \\
             & &  C \ar[ull]^{ \sA \circ \sD}. } 
\end{center}

\vskip 3mm

\begin{remark}
Composing $\AA$ and $\DD$ alternately, and applying it to $F^\dt$
we get six distinct complexes up to translation, corresponding to all 
permutations of the triplet data set $(B,H,C)$. 
In the squarefree setting we thus get a situation of perfect symmetry
between the homological invariants.
In contrast, in the
singly graded case we get a ``symmetry breakdown'' where only $H$ and
$C$ may be transferred into each other by the functor $\DD$, while
the Betti spaces have a distinct position.
\end{remark}
 
\begin{remark} For a positive multidegree $\aa = (a_1, a_2, \ldots, a_n)$,
Alexander duality may also be defined for the more general class of
$\aa$-determined modules, see \cite{MiAlex}. 
(Squarefree modules are $\en$-determined.)
The composition $\AA \circ \DD$
then has order the least common multiple 
$\lcm \{ a_i +2 \, |\, i = 1, \ldots, n \}$, see \cite{BrFl}. In
that paper all the multigraded homology and Betti spaces of 
the iterates $(\AA \circ \DD)^i(S/I)$ are computed for an $\aa$-determined
ideal $I \sus S$. 
\end{remark}

%\begin{remark}
%If $M$ is a squarefree Cohen-Macaulay modules with a pure
%resolution, that is for each $i$ there is a $d_i$ such that
%$B^i_R(M)$ is only nonze

The following will be of particular interest and motivation in the
next Section \ref{PurecxSec}.

\begin{lemma} \label{SetLemCM}
The complexes $\AD(F^\dt)$ and $(\AD)^2(F^\dt)$ are both linear
if and only if $F^\dt$ is a resolution of a Cohen-Macaulay module.
\end{lemma}
%A squarefree module $M$ is a Cohen-Macaulay $S$-module if
%and only if $(\sA \circ \sD)(M)$ and $(\sA \circ \sD)^2 (M)$ are linear
%complexes.

\begin{proof} 
The various homology modules of $F^\dt$ are translated to the various
linear strands of $(\AD)(F^\dt)$.
So $F^\dt$ has only one nonzero homology module iff $(\AD)(F^\dt)$ is linear.
Similarly the cohomology of $F^\dt$ is translated to the Betti spaces
of $(\AD)^2(F^\dt)$ so $F^\dt$ has only one nonzero cohomology module iff
$(\AD)^2 (F^\dt)$ is linear. But the fact that $F^\dt$ has only one 
nonzero homology
module $M$ and $\DD(F^\dt)$ has only one nonzero homology module
is equivalent to $M$ being a Cohen-Macaulay module.
\end{proof}

%The Betti spaces of $(\sA \circ \sD)(M)$ are the 
%homology spaces of $M$. Since the homology of $M$ lives in only one
%homological degree, this corresponds to $(\sA \circ \sD)(M)$ being linear.
%Similarly the Betti spaces of $(\sA \circ \sD)^2 (M)$ are the cohomology 
%spaces of $M$. That $M$ is Cohen-Macaulay means that the cohomology spaces
%live in only one cohomological degree, and this corresponds to 
%$(\sA \circ \sD)^2(M)$ being linear.
%\end{proof}

\subsection{The functor $\sA \circ \sD$ on a basic class of modules}
\label{SetSubsecKoszul}
For $A \sus [n]$, the module $S(-A)$ is a projective module. Denote
by $S/A = S/(x_i)_{i \in A}$. (This is an injective module in $\sq$.)

More generally for a partition $A \cup B \cup C$ of $[n]$, the module
$(S/A)(-B)$ will be a squarefree module. Let us denote it as
$S/A(-B;C)$. These form a basic simple class of squarefree modules
closed with respect to the functors $\sA$ and $\sD$
when we identify modules with their minimal resolutions.

\begin{lemma} \label{SettingLemmaKoszul}
 Let $A \cup B \cup C$ be a partition of $[n]$.
\begin{itemize}
\item[1.] There is a quasi-isomorphism
\[\sD(S/A(-B;C)) \mto{\simeq} S/A(-C;B)[-a].\] 
\item[2.] There is a quasi-isomorphism  
\[\sA (S/A(-B;C)) \mto{\simeq} S/B(-A;C).\]
\end{itemize}
\end{lemma}

\begin{proof}
For $A \sus [n]$ denote by $\kr A$ the vector space generated by $x_i,
i \in A$. The projective resolution of $S/A(-B)$ is 
\[ P^\dt : S(-B) \vpil S(-B) \te (\kr A) \vpil S(-B) \te \wedge^2(\kr A)
\vpil \cdots \vpil S(-B) \te \wedge^a(\kr A). \]
%Then 
%\begin{eqnarray*}
%\Hom_S(S/A(-B;C), \omdot) & \mto{\simeq} &  \Hom_S(P^\dt, \omdot) \\
%& \iso & \Hom_S(P^\dt, S(-\en)). 
%\end{eqnarray*}

The dual complex $\Hom_S(P^\dt, S(-\en))$ is $\DD(S/A(-B))$.
Since the last term $S(-B) \te \wedge^a (\kr A)$ in $P^\dt$ is generated 
in degree $A \cup B$, the dual complex is
\[ S(-C) \vpil S(-C) \te (\kr A) \vpil S(-C) \te \wedge^2(\kr A)
\vpil \cdots \vpil S(-C) \te (\kr A)^a, \]
a resolution of $S/A(-C;B)$. 

To see the second part of the lemma, it is not difficult to verify that
the Alexander dual $\bA(S/A(-B;C)) \iso S/B(-A;C)$.
\end{proof}

We then get the following diagram.

\vskip 3mm
\begin{center}
\hskip -2mm
\xymatrix{  & S/C(-B;A)[-a-c] \ar[dr]^{\sA}  & \\
  S/C(-A;B)[a] \ar[ur]^{\sD} &     & S/B(-C;A)[a+c] \ar[d]^{\sD} \\
S/A(-C;B)[-a] \ar[u]^{\sA} &   & S/B(-A;C)[-a-b-c] \ar[dl]^{\sA[-n]} \\
&    S/A(-B;C) \ar[ul]^{\sD} &  \\  }
\end{center}
%\begin{center}
%Figure \label{SettingFigABC}.
%\end{center}
\vskip 3mm

A particular case is the following diagram.
\vskip 3mm
\begin{center}
\hskip -2mm
\xymatrix{ &  \kk[-n] \ar[dr]^{\sA}  & \\
  \kk \ar[ur]^{\sD} &  & S(-\en)[n] \ar[d]^{\sD} \\
S(-\en) \ar[u]^{\sA} &  & S[-n] \ar[dl]^{\sA[-n]} \\
 & S \ar[ul]^{\sD} &   \\  }
\end{center}
%\begin{center}
%Figure \label{SettingFigSk}. 
%\end{center}
\vskip 3mm

\section{Triplets of pure complexes}

\label{PurecxSec} 
As stated in the introduction
the importance of pure free resolutions of Cohen-Macaulay $S$-modules
is established with the Boij-S\"oderberg conjectures, and
their subsequent demonstration in \cite{ES}.
% They were
%proven in \cite{ES} and \cite{EFW} and both papers contain constructions
%of pure resolutions of Cohen-Macaulay modules.

%Considering a squarefree module $M$, then by Lemma \ref{SettingLemmaRenres}
%$M$ is a Cohen-Macaulay module with pure resolutions if and only if
%$M$ has pure resolution and $(\sA \circ \sD)(M)$ and $(\sA \circ \sD)^2(M)$
%have linear resolutions (but these are not resolutions of a module
%but of a complex).

A complex of free $S$-modules $F^\dt$ is {\it pure}
if it has the form
\[ F^\dt  :   S(-d_0)^{\beta_0} \vpil S(-d_1)^{\beta_1} \vpil
\cdots \vpil S(-d_{r})^{\beta_{r}}  \]
for some integers $d_0 < d_1 < \cdots < d_r$. These integers
are the {\it degree sequence} of the pure complex.

We shall investigate the condition that {\it all three complexes} 
$F^\dt$, $(\sA \circ \sD)(F^\dt)$ 
and $(\sA \circ \sD)^2(F^\dt)$
are {\it pure} when considered as singly graded complexes.
By Lemma \ref{SetLemCM} the special case
that $F^\dt$ is a pure resolution of a Cohen-Macaulay module 
corresponds to the case that $F^\dt$ is pure while 
$(\sA \circ \sD)(F^\dt)$ 
and $(\sA \circ \sD)^2(F^\dt)$ are both linear complexes.
%, corresponds
%to the case that $M$ is a Cohen-Macaulay module with pure resolution.

%The existence and construction of such modules is of 
%centrel importance in Boij-S\"oderberg theory.
%Note by Lemma \ref{SettingLemmaRenres} that the case that $M$ has 
%pure resolution and is a Cohen-Macaulay module corresponds to the
%case when $M$ has pure resolution and $(\sA \circ \sD)(M)$ and 
%$(\sA \circ \sD)^2(M)$
%have linear resolutions (but these are not resolutions of a module
%but of a complex).
%We establish, as in the case of pure resolutions of Cohen-Macaulay
%modules, that the Betti numbers of the three resolutions are completely 
%determined, up to scalar, 
%by the twists $S(-d)$ occuring in these pure resolutions.

\subsection{Basic properties and examples}
We now give an example of a triplet of pure complexes, but let
us first give a lemma telling how $\AD(F^\dt)$ may be computed.

\begin{lemma}
Let 
\[ F^\dt: \cdots \pil F^i \pil F^{i+1} \pil \cdots . \]
Then $\AD(F^\dt)$ is homotopy equivalent to the total complex of 
\[ \cdots \pil \AD(F^i) \pil \AD(F^{i+1}) \pil \cdots . \]
\end{lemma}

\begin{proof} Recall Alexander duality $\bA$ on the category of squarefree
modules.
The complex ${\bA} \circ \DD (F^\dt)$ is simply the complex of modules
\[ \cdots \pil \bA \circ \DD(F^i) \pil \bA \circ \DD(F^{i+1}) \pil \cdots . 
\]
Now if 
\begin{equation} \label{PurecxLigM}
 \cdots \pil M^i \pil M^{i+1} \pil \cdots 
\end{equation}
is a sequence of modules and $F^{i, \dt} \pil M^i$ is a free resolution,
we may lift the differentials $M^i \pil M^{i+1}$, to differentials
$F^{i,\dt} \pil F^{i+1, \dt}$. 
Then the total complex of 
\[ \cdots \pil F^{i,\dt} \pil F^{i+1, \dt} \pil \cdots \]
will be a resolution  of (\ref{PurecxLigM}) and hence it is homotopy 
equivalent to a minimal free resolution of this complex.
Whence the result follows since $\AD(F^i)$ is a free resolution of 
$\bA \circ \DD(F^i)$. 
\end{proof}

\begin{example} \label{PureEksRenecx}
Let $S = \kk[x_1, x_2, x_3]$.
Consider the complex
\[ F^\dt : S \xleftarrow{[x_1x_2, x_1x_3, x_2x_3]} S(-2)^3. \]
First we find $(\sA \circ \sD)(F^\dt)$. 
By the figures of Subsection \ref{SetSubsecKoszul},
$\sA \circ \sD(S)$ is isomorphic to $\kk$, and the 
resolution is the Koszul complex
\[ S \vpil S(-1)^3 \vpil S(-2)^3 \vpil S(-3). \]
(It is really multigraded but for simplicity we only depict it
as singly graded.)
Also $\sA \circ \sD ( S(-([3]\backslash \{i\})))$ is isomorphic
to $S/(x_i)$ and so has resolution $S \vmto{x_i} S(-1)$.

Therefore $\sA \circ \sD(F^\dt)$ is a minimal version of the total 
complex of 
\begin{center}

\xymatrix{S^3 \ar[d]  & S(-1)^3 \ar[l] \ar[d]  &   & \\
 S & S(-1)^3 \ar[l] &  S(-2)^3 \ar[l] & S(-3) \ar[l]. }

\end{center}
It is easily seen that such minimal version is  
is 
\[
S^2 \xleftarrow{\left[ \begin{matrix} x_2x_3 & -x_1x_3 & 0 \\
                                 0 & x_1x_3 & -x_1x_2 
                     \end{matrix} \right ]}
S(-2)^3 \xleftarrow
{\left[ \begin{matrix} x_1 \\ x_2 \\ x_3 \end{matrix} \right ]}
S(-3).
\]
Now consider $(\sA \circ \sD)^2(F^\dt)$. 
By Lemma \ref{SettingLemmaKoszul}, 
$(\sA \circ \sD)^2(S)$ is isomorphic to $S(-3)[3]$ 
and $(\sA \circ \sD)^2 (S(-\{1,2\}))$ is isomorphic to
$S/(x_1, x_2)(-\{3\})[1]$. Therefore $(\sA \circ \sD)^2(F^\dt)$
is a minimal version of the total complex of 

\begin{center}
\xymatrix {S(-1)^3  & S(-2)^6 \ar[l] &  S(-3)^3 \ar[l] \ar[d]\\
             & & S(-3). }
\end{center}

Such a minimal version is then
\[ S(-1)^3 \xleftarrow{\left[ \begin{matrix} x_1 & x_2 & 0 & 0 & 0 & 0 \\
                                       0 & 0 & x_2 & x_3 & 0 & 0 \\
                                       0 & 0 & 0 & 0 & x_3 & x_1 
                     \end{matrix} \right ]}
S(-2)^6 
\xleftarrow{\left[ \begin{matrix} x_2 & -x_1 & -x_3 & x_2 & 0 & 0 \\
                            0 & 0 & x_3 & -x_2 & -x_1 & x_3 
                     \end{matrix} \right ]^{t}}
S(-3)^2.
\]

%\begin{center}
%\xymatrix {S(-1)^3  & S(-2)^6 \ar[l] &  S(-3)^3 \ar[l] \ar[d]\\
%             & & S(-3). }
%\end{center}

%A minimal version of this complex is then
%\[  S(-1)^3 \vpil S(-2)^6 \vpil S(-3)^2. \]
In summary 
\begin{align*}
F^\dt &  : \, S \vpil   S(-2)^3 \\
\sA \circ \sD(F^\dt) &  : \,  S^2 \vpil   S(-2)^3 \vpil S(-3) \\
(\sA \circ \sD)^2(F^\dt) &  : \,  S(-1)^3 \vpil    S(-2)^6 \vpil S(-3)^2.
\end{align*}
So all complexes are pure, and two of them are not linear.
\end{example}

\begin{lemma} \label{PurecxLemStartslutt}
Let 
\[ F^\dt : S(-a_0)^\alpha \vpil S(-a_1)^{\alpha^\prime} \vpil \cdots \]
be a pure complex of squarefree modules with final term $S(-a_0)^\alpha$
in cohomological position $t$. If $\sA \circ \sD (F^\dt)$ also is 
a pure complex, then 
\[ \sA \circ \sD (F^\dt) : \quad \cdots \vpil S(-n+a_0)^\alpha \]
where the initial term $S(-n+a_0)^\alpha$ is in cohomological position 
$-n+a_0+t$.

As a consequence the initial terms of $\DD(F^\dt)$ and its Alexander
dual $\AD(F^\dt)$ are both equal to  $S(-n+a_0)^\alpha$.
\end{lemma}

\begin{proof}
Considered as a complex of graded modules, $\sA \circ \sD (F^\dt)$ is 
the total complex of 

\begin{center}

\xymatrix{ 
S^{\alpha^\prime} \ar[d] & S(-1)^{\alpha^\prime(n-a_1)} \ar[l] \ar[d] & \cdots \ar[l] &  \\
S^\alpha          & S(-1)^{\alpha(n-a_0)} \ar[l] & \cdots \ar[l] &
 S(-n+a_0)^\alpha. \ar[l]
}

\end{center}

When making a minimal complex of the total complex, $S(-n+a_0)^\alpha$ 
cannot cancel out, so it must be the last term. Since $S^\alpha$
is in cohomological position $t$, the last term must be in 
cohomological position $-n+a_0+t$.
\end{proof}

In a pure complex 
\begin{equation} \label{PureLigRenres}
F^\dt  :   S(-a_0)^{\alpha_0} \vpil S(-a_1)^{\alpha_1} \vpil
\cdots \vpil S(-a_{r})^{\alpha_{r}} 
\end{equation}
an integer $d$ is called a {\it degree} of this complex
if $d = a_i$ for some $i$. Otherwise it is called a nondegree.
If the nondegree is in $[a_0, a_r]$ it is an {\it internal}
nondegree.
%We denote by $e = e(F^\dt)$ the number of nontwists in the interval
%$[a_0, a_d]$. 

\medskip
Now suppose we have a situation where $(\sA \circ \sD)^i(F^\dt)$
are pure complexes for $i = 0,1$ and $2$.
Write the complexes as:
\begin{eqnarray*}
F^\dt & : &  S(-a_0)^{\alpha_{0}} \vpil S(-a_1)^{\alpha_{1}} \vpil
\cdots \vpil S(-a_{r_0})^{\alpha_{{r_0}}} \\
\sA \circ \sD(F^\dt) & : & S(-b_0)^{\beta_{0}} \vpil S(-b_1)^{\beta_{1}} \vpil
\cdots \vpil S(-b_{r_1})^{\alpha_{{r_1}}} \\
(\sA \circ \sD)^2(F^\dt) & : & S(-c_0)^{\gamma_{0}} \vpil S(-c_1)^{\gamma_{1}} \vpil
\cdots \vpil S(-c_{r_2})^{\gamma_{{r_2}}}.
\end{eqnarray*}

We denote by $A$ the set of degrees of $F^\dt$, 
and similarly $B$ and $C$ for the degrees of $\AD(F^\dt)$ and 
$(\AD)^2(F^\dt)$. The triplet $(A,B,C)$ is the {\it degree triplet}
of the triplet of pure complexes.
Let $e_A$ be the number of internal nondegrees
of $F^\dt$, and 
correspondingly we define $e_B$ and $e_C$. Let $e$ be the total
number of internal nondegrees for the triplet, $e_A + e_B + e_C$. 
As they turn out to be central invariants, we let 
$c = a_0$, $a=b_0$ and $b = c_0$.

%Also let $e_i$ be the number of 
%nontwists $e((\sA \circ \sD)^i(F^\dt))$ for $i = 0,1,2$, and let
%$e = e_0 + e_1 + e_2$. 

\begin{proposition} \label{PurecxPropDeg}

a. The degrees in the last terms of the complexes above are $a_{r_0} = n - b$,
$b_{r_1} = n - c$, and  $c_{r_2} = n - a$.

b. The number of variables $n = a + b + c + e$. 
\end{proposition}

\begin{proof}
Part a. is by Lemma \ref{PurecxLemStartslutt} above. 
Also, by the lemma above, if $S(-a_0)^{\alpha_{0}}$ in $F^\dt$ is in 
cohomological position $t$, then $S(-b_{r_1})^{\beta_{{r_1}}}$ in 
$\sA \circ \sD (F^\dt)$ is in position $t-b_{r_1}$, and so
the first term $S(-b_0)^{\beta_0}$ is in position $t-b_{r_1} + r_1$.
But $r_1 + e_B = b_{r_1} - b_0$, and so this position is 
$t-a- e_B$. Applying the lemma again, we get that $S(-c_0)^{\gamma_0}$
in $(\sA \circ \sD)^2 (F^\dt)$ is in position $t-a-b -e_B - e_C$.
And then again we get that $S(-a_0)^{\alpha_0}$ in $(\sA \circ \sD)^3 (F^\dt)$
is in position $t-a-b-c-e$.

But since $(\sA \circ \sD)^3$ is isomorphic to the
$n$'th iterate of the translation functor, we get that $n = a+b+c+e$.
\end{proof}

%\begin{proof} We show the middle equality.
%By Lemma \ref{SettingLemmaKoszul} $\sAD(S(-A))$ is 
%$S/A^c$. Then considered as a complex of graded modules, $\sAD(F^\dt)$ is the
%total complex of %
%
%\xymatrix{ 
% & &  \ar[d]
%%& & \ar[d]
% & \ar[d] \\
% &  \cdots \ar[r] & 
%S(-n+a_i)^{\alpha_i} \ar[r] & \cdots \ar[r] & 
%S(-1)^{\alpha_i (n-a_i)} \ar[r] \ar[d] & S^{\alpha_i} \ar[d]\\ 
% & & \vdots \ar[d] & \vdots \ar[d] \\
%S(-n+a_0)^{\alpha_0} \ar[r] & \cdots \ar[r] & 
%S(-n+a_i)^{\alpha_0 \binom{n-a_0}{a_1 - a_0}} \ar[r] & \cdots & 
%S(-1)^{\alpha_0 (n-a_0)} \ar[r] & S^{\alpha_0}
%}
%
%We see that the last term in the total complex will be 
%$S(-n + a_0)^{\alpha_0}$, so $-b_{d_1} = -n+a$. 
%\end{proof}

We can represent the degrees of the complex $F^\dt$ as a string
of circles indexed by the integers from $a_0 = c$ to $a_{r_0} = n-b$
by letting a circle be filled $\disk$ if it is at a position $a_i$ and
be a blank circle $\sirk$ otherwise.
\begin{example}
A complex 
\[ S(-1)^6 \vpil S(-3)^{27} \vpil S(-4)^{24} \vpil S(-7)^3 \]
with $n = 9$, gives rise to the diagram
\[  \overset{1}\disk \vpil \overset{2}\sirk \vpil \overset 3 \disk \vpil
\overset 4 \disk \vpil \overset 5 \sirk \vpil \overset 6 \sirk \vpil
\overset 7 \disk .\]
The dual complex $\sD(F^\dt) = \Hom_S(F^\dt, S(-\en))$, which is 
\[ S(-8)^6 \pil S(-6)^{27} \pil S(-5)^{24} \pil S(-2)^3, \]
gives a diagram by switching the orientation above and letting the
numbering be
 \[  \overset{8}\disk \pil \overset{7}\sirk \pil \overset 6 \disk \pil
\overset 5 \disk \pil \overset 4 \sirk \pil \overset 3 \sirk \pil
\overset 2 \disk .\]
\end{example}

All three complexes may be represented in a triangle, called the
{\em degree triangle} of the three complexes.

% Generated with LaTeXDraw 2.0.8
% Tue Jun 26 14:45:50 CEST 2012
% \usepackage[usenames,dvipsnames]{pstricks}
% \usepackage{epsfig}
% \usepackage{pst-grad} % For gradients
% \usepackage{pst-plot} % For axes
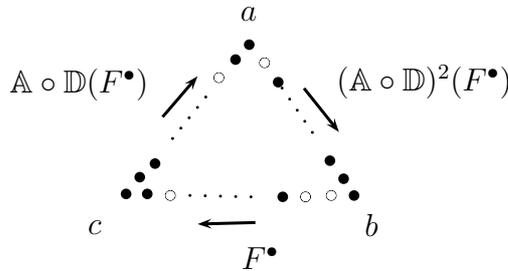
\scalebox{1} % Change this value to rescale the drawing.
{\put(60,-5){
\begin{pspicture}(0,-1.823125)(7.3890624,1.823125)
\psdots[dotsize=0.14](1.79,-0.7703125)
\psdots[dotsize=0.14](2.07,-0.7703125)
\psdots[dotsize=0.14](4.79,-0.7903125)
\psdots[dotsize=0.14](4.65,-0.5703125)
\psdots[dotsize=0.14](3.41,1.2096875)
\psdots[dotsize=0.14](3.23,1.0096875)
\psdots[dotsize=0.14,fillstyle=solid,dotstyle=o](4.49,-0.7903125)
\psdots[dotsize=0.14,fillstyle=solid,dotstyle=o](3.61,0.9696875)
\psdots[dotsize=0.14,fillstyle=solid,dotstyle=o](2.37,-0.7903125)
\psdots[dotsize=0.14](1.97,-0.5503125)
\psdots[dotsize=0.14](2.17,-0.3703125)
\psdots[dotsize=0.14,fillstyle=solid,dotstyle=o](3.01,0.7696875)
\psdots[dotsize=0.14,fillstyle=solid,dotstyle=o](4.15,-0.8103125)
\psdots[dotsize=0.14](4.47,-0.3303125)
\psline[linewidth=0.04cm,linestyle=dotted,dotsep=0.16cm](3.83,0.6496875)(4.25,0.0296875)
\psline[linewidth=0.04cm,linestyle=dotted,dotsep=0.16cm](2.85,0.5296875)(2.39,-0.0503125)
\psdots[dotsize=0.14](3.79,0.7096875)
\psdots[dotsize=0.14](3.85,-0.8103125)
\psline[linewidth=0.04cm,linestyle=dotted,dotsep=0.16cm](2.61,-0.7903125)(3.45,-0.8103125)
\usefont{T1}{ptm}{m}{n}
\rput(3.5545313,-1.6203125){$F^\dt$}
\usefont{T1}{ptm}{m}{n}
\rput(1.3745313,-1.2203125){$c$}
\usefont{T1}{ptm}{m}{n}
\rput(3.4045312,1.6196876){$a$}
\usefont{T1}{ptm}{m}{n}
\rput(5.014531,-1.1803125){$b$}
\psline[linewidth=0.04cm,arrowsize=0.05291667cm 2.0,arrowlength=1.4,arrowinset=0.4]{<-}(2.73,0.7896875)(2.27,0.2496875)
\psline[linewidth=0.04cm,arrowsize=0.05291667cm 2.0,arrowlength=1.4,arrowinset=0.4]{<-}(2.73,-1.1703125)(3.49,-1.1503125)
\psline[linewidth=0.04cm,arrowsize=0.05291667cm 2.0,arrowlength=1.4,arrowinset=0.4]{<-}(4.59,0.0496875)(4.13,0.6296875)
\usefont{T1}{ptm}{m}{n}
\rput(1.1745312,0.6796875){$\AD (F^\dt)$}
\usefont{T1}{ptm}{m}{n}
\rput(5.7145314,0.7196875){$(\AD)^2 (F^\dt)$}
\end{pspicture} 
}}

%\medskip
%\vskip 4cm
%\medskip

\noindent Note that the degrees of $F^\dt$ starts with $c$, then
proceed in ascending order and ends with $n-b$. 

\begin{lemma} \label{PurecxLemabc}
In the degree triangle above the following holds.
\begin{itemize}
\item[a.] The length, i.e. the number of circles (both filled and blank), 
of the side
corresponding to the set $A$, the degree sequence of $F^\dt$, is $a + e+1$.
Similar relations hold for the other sides.
\item[b.] The number of circles in the degree triangle is 
$a+b+c + 3e$. In particular at most a third of the circles are blank
circles.
\item[c.]
The Koszul complexes given in 
Subsection \ref{SetSubsecKoszul} give all cases of degree triangles
where there are no internal nondegrees, i.e. no blank circles.
\end{itemize}
\end{lemma}

\begin{proof}
a. This is because the number of circles is the cardinality of the interval
$[c, n-b]$ which is this number by Proposition \ref{PurecxPropDeg}.
Part b. above follows immediately from part a. Regarding part c. 
there are three numerical parameters for these Koszul complexes,
the cardinalities of $|A|, |B|$ and $|C|$, and these correspond to 
$a+1,b+1$ and $c+1$.
\end{proof}

%\noindent{\bf Observation.}
%The length, i.e. the number of circles, of the side
%corresponding to the set $A$, the degree sequence of $F^\dt$, is $a + e+1$.
%This is because the number of circles is the cardinality of the interval
%$[c, n-b]$ which is this number by Lemma \ref{PurecxPropDeg}.
%Similar relations hold of course for the other two sides.

%\noindent where the first complex $F^\dt$ is represented by the arc from $c$ 
%to
%$a$ oriented clockwise. The numbering should then be in descending order
%and be from $n-c$ to $a$, while if we consider $\sD (F^\dt)$ one should
%proceed counterclockwise from $a$ to $c$ and the numbering should be
%in decreasing order from $n-a$ to $c$.

%Similarly $\sAD(F^\dt)$ is represented by the clockwise arc from
%$a$ to $b$, and $\sAD^2(F^\dt)$ by the clockwise arc from $b$ to $c$.

\begin{example}
The minimal complexes in Example \ref{PureEksRenecx}
give rise to the following degree triangle.

% Generated with LaTeXDraw 2.0.8
% Tue Jun 26 14:53:07 CEST 2012
% \usepackage[usenames,dvipsnames]{pstricks}
% \usepackage{epsfig}
% \usepackage{pst-grad} % For gradients
% \usepackage{pst-plot} % For axes
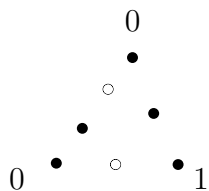
\scalebox{1} % Change this value to rescale the drawing.
{\put(120,-15){
\begin{pspicture}(0,-1.253125)(3.0690625,1.253125)
\psdots[dotsize=0.144](0.79,-0.8403125)
\psdots[dotsize=0.144](2.39,-0.8603125)
\psdots[dotsize=0.144](1.79,0.5596875)
\psdots[dotsize=0.144](2.07,-0.1803125)
\psdots[dotsize=0.144](1.13,-0.3803125)
\psdots[dotsize=0.144,fillstyle=solid,dotstyle=o](1.57,-0.8603125)
\psdots[dotsize=0.144,fillstyle=solid,dotstyle=o](1.47,0.1396875)
\usefont{T1}{ptm}{m}{n}
\rput(1.7945312,1.0496875){$0$}
\usefont{T1}{ptm}{m}{n}
\rput(0.27453125,-1.0503125){$0$}
\usefont{T1}{ptm}{m}{n}
\rput(2.6945312,-1.0503125){$1$}
\end{pspicture} 
}}
\end{example}
\vskip 3mm

\subsection{A balancing condition}
Suppose $F^\dt, \sAD (F^\dt)$ and $(\sAD)^2(F^\dt)$ 
is a triplet of pure free squarefree complexes. The interior
nondegrees of these complexes cannot be arbitrarily distributed.
There is a certain balancing condition which we now give.

%If $G^\dt$ is a pure complex, let the linear strands be 
%\[ \Gd_{\langle i_r \rangle}, \Gd_{\langle i_{r-1} \rangle}, \ldots, 
%\Gd_{\langle i_0 \rangle} \]
%where $i_r < i_{r-1} < \cdots < i_0$, so the final term of 
%$\Gd_{\langle i_j \rangle}$ maps to the initial term of $\Gd_{\langle i_{j+1} \rangle}$.

%Let $\Xd$ be a complex of $S$-modules, write a nonzero homology module as
%\[ H^i(\Xd) = \oplus_{j = m_i}^{M_i} H^i(\Xd)_j \]
%where $m_i$ (resp. $M_i$) is the smallest (resp. largest) nonzero degree
%of $H^i(\Xd)$. (If the homology group is zero, these are undefined.)
%Recall that if $\Gd$ is a resolution of $\sAD(\Xd)$, then
%\[ G_{\langle i \rangle} = \gL(H^i(\Xd))[n-i] \]
%so the nonzero homology groups occur when $i$ is 
%$i_0 >  i_1 >  \cdots > i_r$. 

%\begin{lemma}
%When the resolution $\Gd$ of $\sAD(\Xd)$ is pure, then
%\[ m_{i_{j+1}} = M_{i_j} + i_{j+1} - i_j + 1. \]
%\end{lemma}

%\begin{proof}
%The final term of $G_{\langle i_j \rangle}$ is in cohomological position
%$M_{i_j} + i_j - n$. The initial term of $G_{\langle  i_{j+1} \rangle}$ is 
%in cohomological position $m_{i_{j+1}} + i_{j+1} - n$. But the latter
%position must be one more than the former, which gives the relation.
%\end{proof}

Let $G^\dt$ be one of the three complexes, 
so $G^\dt$ and its Alexander dual $\AA(G^\dt)$ are pure complexes. 
In particular they have the same initial term $S(-n+g)^\gamma$.
We can display their degrees as

% Generated with LaTeXDraw 2.0.8
% Wed Jul 04 20:49:19 CEST 2012
% \usepackage[usenames,dvipsnames]{pstricks}
% \usepackage{epsfig}
% \usepackage{pst-grad} % For gradients
% \usepackage{pst-plot} % For axes
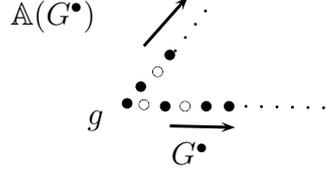
\scalebox{1} % Change this value to rescale the drawing.
{\put(100,0){
\begin{pspicture}(0,-1.1664063)(4.73,1.1464063)
\psdots[dotsize=0.14](2.11,-0.29359376)
\psdots[dotsize=0.14](2.29,-0.07359375)
\psdots[dotsize=0.14](2.67,0.34640625)
\psdots[dotsize=0.14](2.61,-0.33359376)
\psdots[dotsize=0.14](3.15,-0.33359376)
\psdots[dotsize=0.14,fillstyle=solid,dotstyle=o](2.33,-0.31359375)
\psdots[dotsize=0.14,fillstyle=solid,dotstyle=o](2.87,-0.33359376)
\psdots[dotsize=0.14,fillstyle=solid,dotstyle=o](2.51,0.12640625)
\psdots[dotsize=0.14](3.45,-0.33359376)
\psline[linewidth=0.04cm,linestyle=dotted,dotsep=0.16cm](3.65,-0.33359376)(4.71,-0.35359374)
\psline[linewidth=0.04cm,linestyle=dotted,dotsep=0.16cm](2.73,0.38640624)(3.15,0.94640625)
\psline[linewidth=0.04cm,arrowsize=0.05291667cm 2.0,arrowlength=1.4,arrowinset=0.4]{<-}(2.91,1.1264062)(2.33,0.46640626)
\psline[linewidth=0.04cm,arrowsize=0.05291667cm 2.0,arrowlength=1.4,arrowinset=0.4]{<-}(3.53,-0.61359376)(2.67,-0.5935938)
\usefont{T1}{ptm}{m}{n}
\rput(2.9345312,-0.9635937){$G^\dt$}
\usefont{T1}{ptm}{m}{n}
\rput(1.1345313,0.8564063){$\AA(G^\dt)$}
\usefont{T1}{ptm}{m}{n}
\rput(1.6945312,-0.5235937){$g$}
\end{pspicture} 
}}

The balancing condition is the following.

\begin{proposition} \label{PurecxPropBalance}
Suppose $G^{\dt}$ belongs to a triplet
of pure free squarefree complexes and let $S(-n+g)^\gamma$ be the
initial term of $G^\dt$
and its Alexander dual  $\AA(G^\dt)$.
Then for each $0  \leq v \leq n-g$, the number of degrees of $G^\dt$ 
in the interval $[v, n-g]$ is greater than the number
of nondegrees of $\AA(G^\dt)$ in the interval $[v, n-g]$. 
\end{proposition}

\begin{proof}
We may let $F^\dt = \DD (G^\dt)$, so $F^\dt, \AD(F^\dt)$, and 
$(\AD)^2(F^\dt)$ is a triplet of pure free squarefree complexes.
With this notation
\[G^\dt = \DD(F^\dt): S(-n+c)^{\alpha} \pil \cdots \pil S(-b)^{\alpha^\prime} \]
so 
%\[ F^\dt = \DD(G^\dt): S(-c)^{\alpha} \vpil \cdots \vpil 
%S(-n+b)^{\alpha^\prime}. \] 
\[\AA(G^\dt) = \AD (F^\dt): S(-n+c)^{\alpha} \pil \cdots 
S(-a)^{\alpha^{\prime\prime}}.\]
Let $\phi(v)$ be the sum of the number of degrees of $G^\dt$ in 
$[v,n-c]$ and the number of degrees of $\AA(G^\dt)$ in this interval.
The statement of the proposition is equivalent to: 
$\phi(v)$ is greater than than the cardinality of $[v,n-c]$.

%The number of degrees of $G^\dt$ in
%$[v, n-c]$ pluss the number of degrees of $\AA(G^\dt)$ in $[v, n-c]$
%is greater than the cardinality of $[v, n-c]$. 

%Note that the last term of $G^\dt$ is $S(-b)^{\alpha^\prime}$ and the 
%last term of 
%$\AA(G^\dt)$ is $S(-a)^{\alpha^{\prime\prime}}$. 

\medskip
\noindent {\it Case 1.} In the range $0 \leq v \leq \max \{a, b\}$
the difference $\phi(v) - |[v,n-c]|$ is weakly decreasing. So in order to prove
the statement in this range, it is enough to prove that 
\begin{equation} \label{PurecxLigNull}
 \phi(0) > |[0,n-c]| = n-c+1 = a+b+e+1.
\end{equation}
But 
\begin{align*}
  \phi(0) & =  |A| + |B| \\
 & =  a+e+1-e_A +b+e+1 - e_B \\
& =  a+b+2 + e + e_C,
\end{align*}
so clearly (\ref{PurecxLigNull}) holds.

\medskip
\noindent {\it Case 2.} Suppose now $v > \max \{ a,b \}$. 
We may as well assume that $a \geq b$, so $v > a$.
Let $c = a_0, a_1, \ldots$ be the degrees of   $F^\dt = \DD(G^\dt)$,
with $S(-a_i)^{\alpha_i}$ in cohomological degree $-i$. 
The cohomology module $H^{-i}(F^\dt)$ is transferred to 
the $-i$'th linear strand of $\AA(G^\dt)$. Note that 
if $i > 0$, the least nonzero degree of this cohomology module,
if this module is nonzero,
is $\geq a_i + 1$. Hence the largest degree occurring
in the $-i$'th linear strand of $\AA(G^\dt)$, if this strand
is nonzero, is $ \leq n-a_i -1$.

If $v$ is a degree of $\AA(G^\dt)$, let $-l$ be the linear strand
it belongs to. Since $n-c$ belongs to the $0$'th linear strand,
the number of degrees of $\AA(G^\dt)$ in $[v,n-c]$ is $n-c-v+1-l$. 

Now if $v-1$ is an interior nondegree of $\AA(G^\dt)$
then 
\[ \phi(v-1) - |[v-1,n-c]| \leq \phi(v) - |[v,n-c]|.  \]
Therefore we might as well prove the statement for $v-1$. 
Since $a$ is a degree of $\AA(G^\dt)$, we may continue
this way and in the end come to a situation where $v-1$ is a degree
of $\AA(G^\dt)$. Let $-l$ be its linear strand in 
$\AA(G^\dt)$.  When $l > 0$, by what said above, $v-1 \leq n-a_l -1$ or 
equivalently $a_l \leq n-v$. But this also holds when $l = 0$. 
Hence the degrees $a_0, a_1, \ldots, a_l$ of 
$F^\dt$ all belong to $[c,n-v]$ and so
\[ \phi(v) \geq (n-c-v+1-l) + (l+1) > n-c-v+1. \]
\end{proof}

%Let $c = a_0, \ldots, a_q$ be the degrees of $F^\dt$ contained in the interval
%$[c, n-v]$. The $i$'th homology of $F^\dt$ goes to the $i$'th linear
%strand of $\AD(F^\dt)$. Since the $(q+1)$'th homology of $F^\dt$ has
%its generator of least degree in degree $\geq a_{q+1} > n-v$, the
%$(q+1)$'th linear strand of $\AD(F^\dt)$ has initial term starting in
%degree $< v$. (By construction of $\gL$.) Hence in the interval 
%$[v, n-c]$ at most the $0$'th, ..., $q$'th linear strand
%can occur (some may be $0$). Thus the number of degrees of $\AD(F^\dt)$
%in this interval is $\geq n-c - v - q+1$. Since the number of 
%degrees of $G^\dt = \DD(F^\dt)$ in $[v, n-c]$ is $q+1$, we get the 
%equivalent statement given above.

Given a natural number $n$. 
For an integer $d$ let $\overline{d} = n-d$, and for a subset of integer
$D$ let $\overline{D} = \{ \overline{d} \, | \, d \in D \}$.

\begin{definition}
A triplet of nonempty 
subsets $(A,B,C)$ of ${\mathbb N}_0$ is a {\it balanced degree 
triplet of type $n$} if there are integers $0 \leq a,b,c, \leq n$ such
that 
\begin{itemize}
\item[1.] 
\[ A \sus [c, \overline{b}],\,\,  B \sus [a, \overline{c}], \,\, 
C \sus [b, \overline{a}] \]
and the endpoints of each interval are in the respective subsets $A,
B$ or $C$. 
\item[2.] Let $e_A$ be the cardinality of $[c,\overline{b}] \backslash A$ and
correspondingly define $e_B$ and $e_C$. Then $n = a+b+c+e_A + e_B + e_C$. 
\item[3.] $A$ and $\overline{B}$ are balanced with respect to the
common endpoint $c$, i.e. for each $c \leq v \leq n$, the number
of elements of $[c,v]$ in $A$ is greater than the number of elements
of $[c,v]$ not in $\overline{B}$. Similarly for $B$ and $\overline{C}$
with respect to $a$ and $C$ and $\overline{A}$ with respect to $b$. 
\end{itemize}
\end{definition}

% Generated with LaTeXDraw 2.0.8
% Wed Sep 12 14:27:56 CEST 2012
% \usepackage[usenames,dvipsnames]{pstricks}
% \usepackage{epsfig}
% \usepackage{pst-grad} % For gradients
% \usepackage{pst-plot} % For axes
\scalebox{1} % Change this value to rescale the drawing.
{\put(120,-15)
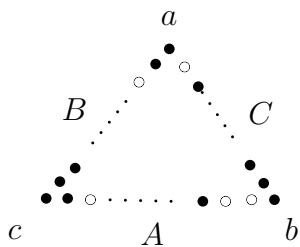
{
\begin{pspicture}(0,-1.633125)(4.2690625,1.633125)
\psdots[dotsize=0.14](0.67,-0.9603125)
\psdots[dotsize=0.14](0.95,-0.9603125)
\psdots[dotsize=0.14](3.67,-0.9803125)
\psdots[dotsize=0.14](3.53,-0.7603125)
\psdots[dotsize=0.14](2.29,1.0196875)
\psdots[dotsize=0.14](2.11,0.8196875)
\psdots[dotsize=0.14,fillstyle=solid,dotstyle=o](3.37,-0.9803125)
\psdots[dotsize=0.14,fillstyle=solid,dotstyle=o](2.49,0.7796875)
\psdots[dotsize=0.14,fillstyle=solid,dotstyle=o](1.25,-0.9803125)
\psdots[dotsize=0.14](0.85,-0.7403125)
\psdots[dotsize=0.14](1.05,-0.5603125)
\psdots[dotsize=0.14,fillstyle=solid,dotstyle=o](1.89,0.5796875)
\psdots[dotsize=0.14,fillstyle=solid,dotstyle=o](3.03,-1.0003124)
\psdots[dotsize=0.14](3.35,-0.5203125)
\psline[linewidth=0.04cm,linestyle=dotted,dotsep=0.16cm](2.71,0.4596875)(3.13,-0.1603125)
\psline[linewidth=0.04cm,linestyle=dotted,dotsep=0.16cm](1.73,0.3396875)(1.27,-0.2403125)
\psdots[dotsize=0.14](2.67,0.5196875)
\psdots[dotsize=0.14](2.73,-1.0003124)
\psline[linewidth=0.04cm,linestyle=dotted,dotsep=0.16cm](1.49,-0.9803125)(2.33,-1.0003124)
\usefont{T1}{ppl}{m}{n}
\rput(0.25453126,-1.4103125){$c$}
\usefont{T1}{ppl}{m}{n}
\rput(2.2845314,1.4296875){$a$}
\usefont{T1}{ppl}{m}{n}
\rput(3.8945312,-1.3703125){$b$}
\usefont{T1}{ptm}{m}{n}
\rput(2.0645313,-1.4303125){$A$}
\usefont{T1}{ptm}{m}{n}
\rput(3.4945312,0.1696875){$C$}
\usefont{T1}{ptm}{m}{n}
\rput(1.0445312,0.2296875){$B$}
\end{pspicture} 
}}

\begin{remark} Note that parts a. and b. of Lemma \ref{PurecxLemabc}
may be deduced solely from the properties 1. and 2. above.
\end{remark}

\begin{conjecture} \label{PurecxConj}

a. For each balanced degree triplet $(A,B,C)$ of type $n$, there exists
a triplet of pure free squarefree complexes over the polynomial ring
in $n$ variables whose degree sequences are given by 
$A$, $B$, and $C$. 

b. The Betti numbers of this triplet of complexes are uniquely determined by the
degree triplet, up to common scalar multiple.
\end{conjecture}

\section{Constraints on the Betti numbers}

\label{NumSec}

In this section we give linear equations fulfilled by the
Betti numbers in a triplet of pure complexes. The number of equations
is one less than the number of Betti numbers, so we expect
a unique set of Betti numbers up to multiplication by a common scalar.
We prove that this is the case, provided part a. of Conjecture
\ref{PurecxConj} holds. In other words we prove that part a. of
the conjecture implies part b.

\subsection{Some elementary relations for binomial coefficients}
For nonnegative integers $p$ we have the binomial coefficient
$\binom{x}{p}$. When $p$ is a negative integer we set this coefficient
to be zero. The following identities hold in ${\mathbb Q}[x,y]$
and are repeatedly used in the proof of the below lemma.
%\noindent{\bf Facts.}
\begin{itemize}
\item[1.] $\binom{x+y}{p} = \sum_{i = 0}^p \binom{x}{p-i} \binom{y}{i}$,
\cite[Ex. 4.3.3]{CoRo}.
\item[2.] $\binom{x}{p} = (-1)^p \binom{p-1-x}{p}$.
\end{itemize}

Let  $A = (a_{ij})$ be the $(n+1) \times (n+1)$-matrix
with $a_{ij} = (-1)^j \binom{n-j}{i}$ for $i,j = 0, \ldots, n$.
For instance when $n = 2$ this is the matrix 
\[ \left [ \begin{matrix}
1 & -1 & 1 \\
2 & -1 & 0 \\
1 & 0 & 0 \\

\end{matrix} \right ].
\]

\begin{lemma} \label{NumLemAAA}
$A^3 = (-1)^n \cdot I$
\end{lemma}

\begin{proof}
First we show that $A^2 = (b_{ij})$ where $b_{ij} = (-1)^j \binom{j}{n-i}$.
For instance when $n = 2$ this is 
\[ \left [ \begin{matrix}
0 & 0 & 1 \\
0 & -1 & 2 \\
1 & -1 & 1
\end{matrix} \right ].
\]
The $i$'th row in $A$ is
\begin{equation*}\binom{n}{i}, \,\, - \binom{n-1}{i}, \, \, \binom{n-2}{i},\,  \cdots
\end{equation*} 
Now 
\[ \binom{n-j}{i} = \binom{n-j}{n-j-i} = (-1)^{n-i-j}\binom{-i-1}{n-j-i}. \]
The $i$'th row of $A$ is then $(-1)^{n-i}$ multiplied with:
\[ \binom{-i-1}{n-i}, \, \, \binom{-i-1}{n-1-i}, \, \, \binom{-i-1}{n-2-i}, \, \cdots .\]
The $j$'th column in $A$ is $(-1)^j$ multiplied with the following: 
\[ \binom{n-j}{0}, \,\, \binom{n-j}{1}, \,\, \binom{n-j}{2}, \, \cdots \]
From this $b_{ij}$ is $(-1)^{n-i-j}$ multiplied with
\begin{equation} \binom{n-j}{0} \binom{-i-1}{n-i} +  \binom{n-j}{1}\binom{-i-1}{n-1-i} + \cdots 
\end{equation}
which by property 1. in the beginning of this section equals
\begin{equation}
\binom{n-j-i-1}{n-i} = (-1)^{n-i} \binom{j}{n-i}.
\end{equation}
Hence $b_{ij} = (-1)^j \binom{j}{n-i}$.

\medskip
To find $A^3$ note that the $i$'th row in $A^2$ is 
\[ \binom{0}{n-i},\,\, -\binom{1}{n-i}, \,\, \binom{2}{n-i}, \, \cdots  .\]
Note that
\[ \binom{j}{n-i} = \binom{j}{j+i-n} = (-1)^{j+i-n}\binom{i-n-1}{j+i-n}.\]
Hence row $i$ is $(-1)^{n-i}$ multiplied with 
\[ \binom{i-n-1}{i-n}, \, \, \binom{i-n-1}{i-n+1}, \, \, \binom{i-n-1}{i-n+2}, \, \cdots . \]

The $j$'th column in $A$ is $(-1)^j$ multiplied with
\[ \binom{n-j}{0}, \, \, \binom{n-j}{1}, \, \, \binom{n-2}{2}, \, \cdots . \]
Since 
\[ \binom{n-j}{i} = \binom{n-j}{n-j-i} \]
this column becomes
\[\binom{n-j}{n-j}, \binom{n-j}{n-j-1}, \binom{n-j}{n-j-2}, \cdots .\]

The first nonzero position in the $i$'th row is $n-i$. The last
nonzero position in the $j$'th column is $n-j$. 
Hence if $n-j < n-i$, equivalently $i < j$, the product of
the $i$'th row and $j$'th column is zero. 
On the other hand if $i \geq j$ the product is 
$(-1)^{n-i-j}$ multiplied with 
\[ \binom{i-1-j}{i-j} = (-1)^{i-j}\binom{0}{i-j} 
= \begin{cases} 1 & i = j \\ 0 & i > j
\end{cases}. \]
Hence we obtain $A^3 = (-1)^n \cdot I$. 
\end{proof}

\subsection{Linear equations for the Betti numbers}
Let $F^\dt$ be the pure free squarefree complex
\begin{equation} \label{NumLigF}
 F^\dt: S(-a_0)^{\alpha_{0}} \vpil S(-a_1)^{\alpha_{1}} \vpil \cdots
\vpil S(-a_r)^{\alpha_{r}}.
\end{equation}
Let $\hal_{a_i} = (-1)^{l(a_i)}
\cdot \alpha_{i}$ where $l(a_i)$ is the linear strand containing the term
$S(-a_i)^{\alpha_i}$, be the {\it signadjusted} Betti numbers. We set
$\hal_d  = 0$ if $d$ is not a degree of $F^\dt$. Note that 
these signadjusted Betti numbers are parametrized by the internal degrees.
Note also that 
\[ (-1)^{l(a_i)} = (-1)^{i + a_i + l(a_0) - a_0}. \]

Assume also that $\AA \circ \DD (F^\dt)$ is a pure complex
\[ S(-b_0)^{\beta_{0}} \vpil \cdots \vpil S(-b_{r^\prime})^{\beta_{{r^\prime}}}. \]
Recall  that the
$i$'th homology module of $F^\dt$ is transferred to the
$i$'th linear strand of $\AD (F^\dt)$.

Suppose the $i$'th homology module of $F^\dt$ is nonzero and let $d$ be 
a degree for which the $d$'th graded part of this module is nonzero.
This module is squarefree and the dimension of its squarefree part in 
degree $d$ (recall this notion in Subsection \ref{SetSubsecSqf}) is 
\begin{equation} \label{NumLigSqfri}
(-1)^{i+l(a_0) - a_0}[\alpha_{a_0} \binom{n-a_0}{n-d} - \alpha_{a_1}\binom{n-a_1}{n-d}
+ \alpha_{a_2} \binom{n-a_2}{n-d} + \cdots ].
\end{equation}
This will be equal to $(-1)^{i}\hbe_{n-d}$.
% Also
%\[ \hbe_{n-d} = (-1)^{b_0 + e_B - i} \beta_{n-d},\]
The equations (\ref{NumLigSqfri}) when $d$ varies are then the same as
\begin{equation} \label{NumLigAtilB}
\hbe = A \cdot \hal.
\end{equation}
If furthermore $(\AD)^2 (F^\dt)$ is pure we get in the same way
\begin{eqnarray} \label{NumLigBtilC}
\hat{\gamma} &  = & A \cdot \hat{\beta} 
\\ \label{NumLigCtilA}
\hat{\hat{\alpha}} & = & A \cdot \hat{\gamma},
\end{eqnarray}
where $\hat{\hat{\alpha}} = (-1)^n \hal$ due to the shift $n$ of linear
strands of the functor $(\AD)^3$. 

If $F^\dt, \AD (F^\dt)$ and $(\AD)^2 (F^\dt)$ are
all pure then clearly the following equations hold

\begin{align} \label{NumLigAo}
&\hat{\alpha_i} = 0  \mbox { for all nondegrees $i$ of } 
F^\dt  \mbox{ in } [0,n], \\
\label{NumLigBo} &\hat{\beta_i} = 0  \mbox { for all nondegrees $i$ of }
(\AD) (F^\dt) 
\mbox{ in } [0,n], \\
\label{NumLigCo} &\hat{\gamma_i} = 0  \mbox { for all nondegrees $i$ of } 
(\AD)^2 (F^\dt) 
\mbox{ in } [0,n].
\end{align}
In addition we must have the equations 
(\ref{NumLigAtilB}), (\ref{NumLigBtilC}) and 
(\ref{NumLigCtilA}) above (where any two of these
determine the third by Lemma \ref{NumLemAAA}).

\begin{lemma} \label{NumLemEkviv}
The equations $\hal_i = 0$ for $i = 0, \ldots, c-1$
are equivalent to the equations $\hbe_{n-i} = 0$ for $i = 0, \ldots, c-1$. 

Similarly the equations $\hbe_i=0$ and $\hga_{n-i} = 0$
for $i = 0, \ldots, a-1$ are equivalent, and $\hga_i = 0$
and $\hal_{n-i} = 0$ for $i = 0, \ldots, b-1$ are equivalent.
\end{lemma}

\begin{proof}
This is due to the transition matrix $A$ having the triangular
form 
\[ \left \lfloor \begin{matrix} \cdot & \cdot & \cdot & \cdot & \cdot \\  
                           * & * & * & 0 & \cdots \\
                           * & * & 0 & \cdots & \\
                           * & 0 & \cdots & & \\
           \end{matrix} \right \rfloor.
\]
\end{proof}

\begin{corollary}
Given a balanced degree triangle. The $3n+3$ signadjusted Betti numbers
$\hat{\alpha_i}, \hat{\beta_i}, \hat{\gamma_i}$, $i = 0, \ldots, n$
fulfill the equations (\ref{NumLigAtilB}), (\ref{NumLigBtilC}), 
(\ref{NumLigCtilA}), (\ref{NumLigAo}), (\ref{NumLigBo}), 
and (\ref{NumLigCo}), which may be reduced to $3n+2$ natural equations.
\end{corollary}

\begin{remark} We expect these equations to be linearly independent. Hence there
would be a unique solution up to scalar multiple.
\end{remark}

\begin{proof}
There are $c+b+e_A$ equations of the form $\hal_i = 0$. Similarly
there are $a+c+e_B$ equations of the form $\hbe_i = 0$, and 
$a+b+e_C$ equations of the form $\hga_i = 0$. This give a total
of $2a+2b+2c + e$ equations. However by the above Lemma \ref{NumLemEkviv} there are
$a+b+c$ dependencies among them, giving $a+b+c + e = n$ equations.
In addition the transition equations (\ref{NumLigBtilC}) and (\ref{NumLigCtilA})
give $2n+2$ further equations, a total of $3n+2$. 
\end{proof}

The complex $F^\dt$ is 
\[ S(-n+b)^{\alpha_{r_0}} \pil \cdots \pil S(-c)^{\alpha_0}. \]
Its Alexander dual
%given in (\ref{NumLigF})
%where $a_0 = c$ and $a_r = n-b$ starts $S(-n+b)^{\alpha_{a_r}} \vpil
%\cdots $. The complex 
$\AA (F^\dt)$ 
equals (up to translation) $\DD \circ
(\AD)^2 F^\dt$ which is 
\[ S(-n+b)^{\gamma_{0}} \pil \cdots \pil S(-a)^{\gamma_{r_2}}. \]
(Note that $\gamma_0 = \alpha_{r_0}$.) 
Let $v_1 < \cdots < v_{e_C}$ be the internal nondegrees of $\AA (F^\dt)$.

The complex $\DD (F^\dt)$ is 
\[ S(-n+c)^{\alpha_0} \pil \cdots \pil S(-b)^{\alpha_{r_0}} \]
and then its Alexander dual $\AD (F^\dt)$ is 
\[ S(-n+c)^{\beta_{r_1}} \pil \cdots \pil S(-a)^{\beta_0}. \]
(Note that $\beta_{r_1} = \alpha_0$.) 
Let $u_1 < \cdots < u_{e_B}$ be the internal nondegrees of $\AD (F^\dt)$.

\begin{proposition} \label{NumProAsystem}
Given a triplet of pure
free squarefree complexes.  
Let $a_0 < \cdots < a_r$ be the degrees of the first complex $F^\dt$ and 
$\overline{a}_r < \cdots < \overline{a}_0$ be the degrees of the dual 
$\DD (F^\dt)$. 
By the transition 
equations (\ref{NumLigAtilB}), (\ref{NumLigBtilC}), and (\ref{NumLigCtilA}), 
the equations (\ref{NumLigAo}), (\ref{NumLigBo}), and (\ref{NumLigCo}) are 
equivalent to
the following equations for the $(r+1)$ nonzero Betti numbers $\alpha_{i}$. 
\begin{align} \label{NumLigAen}
\alpha_{0} \binom{a_0}{v_i} - & \alpha_{1} \binom{a_1}{v_i} + \cdots + 
(-1)^r \alpha_{r} \binom{a_r}{v_i}  = 0, & i = 1, \ldots, e_C \\
\label{NumLigAto}
\alpha_{0} \binom{\overline{a}_0}{u_i} - 
& \alpha_{1} \binom{\overline{a}_{1}}{u_i} + \cdots + 
(-1)^r \alpha_{r} \binom{\overline{a}_r}{u_i} = 0, & i = 1, \ldots, e_B \\
\label{NumLigAtre}
\alpha_{0} \binom{a_0}{j} - & \alpha_{1} \binom{a_1}{j} + \cdots + 
(-1)^r \alpha_{r} \binom{a_r}{j} = 0, &  j = 0, \ldots, a-1 
\end{align}
The total number of these equations $e_C + e_B + a$, equals $r$.
\end{proposition}

\begin{proof}
The last part is because 
\[ r + e_A = n-b-c, \mbox{ and } a+b+c + e_A + e_B + e_C = n. \]

By the transition equation (\ref{NumLigAtilB}) the set of equations 
(\ref{NumLigAto})
is equivalent to $\hbe_{u_i} = 0$ for each nondegree $u_i$ of $\AD(F^\dt)$
in the interval $[a,n-c]$.
The vanishing of $\hbe_j$ for $j \in [n-c+1,n]$ is by Lemma \ref{NumLemEkviv}
equivalent to $\hal_j = 0$ for $j \leq c-1$. 
The vanishing of $\hbe_j$ for $j \in [0, a-1]$ is again by the transition 
equation (\ref{NumLigAtilB}) equivalent to the equations (\ref{NumLigAtre}).
  
In the same way the vanishing of $\hga_j$ for each nondegree $j$ of 
$(\AD)^2(F^\dt)$ in the interval $[b, n-a]$ is equivalent to the equations
(\ref{NumLigAen}). The vanishing of $\hga_j$ for $j$ in $[n-a+1, n]$ are
by Lemma \ref{NumLemEkviv} equivalent to the vanishing of $\hbe_j$ for 
$j$ in $[0, a-1]$ which are again equivalent to equations (\ref{NumLigAtre}).

%The first set of equations are equivalent to $\gamma_{v_i} = 0$ for
%each of the nondegrees $v_i$ of $\AA F^\dt$. The second set of equations
%is equivalent to $\beta_{u_i} = 0$ for each of the nondegrees
%of $u_i$ of $\AD F^\dt$. The third set of equations are by Lemma
%XX equivalent to the vanishing of the $\hbe_i$ for $i \leq a-1$
%or the vanishing of the $\hga_i$ for $i \leq a-1$. 
\end{proof}

We also get corresponding equations for the $\beta_i$ and the $\gamma_i$.

\begin{corollary}
Given a balanced degree triplet. If the equations
(\ref{NumLigAen}), (\ref{NumLigAto}), and (\ref{NumLigAtre})
for the Betti numbers
$\alpha_i$ of the pure complex $F^\dt$
has a $k$-dimensional solution set, then the corresponding equations for
the Betti numbers $\beta_i$ of the pure complex $\AD (F^\dt)$ has a 
$k$-dimensional solution set, and similarly  for the
Betti numbers $\gamma_i$ of $(\AD)^2 (F^\dt)$.
\end{corollary}

\begin{proof}
By the transition
equations (\ref{NumLigAtilB}), (\ref{NumLigBtilC}), and (\ref{NumLigCtilA}), 
all these equation systems are equivalent to
the equations (\ref{NumLigAo}), (\ref{NumLigBo}), and (\ref{NumLigCo}).
%The equation set for the $\alpha_i$ are equivalent, via the transformations
%(\ref{NumLigAlGa}) and (\ref{NumLigGaBe}) to the following equations:
%\begin{itemize}
%\item $\hal_i = 0$ for each nondegree $i$ of $F^\dt$ in $[0,n]$,
%\item $\hbe_i = 0$ for each nondegree $i$ of $\AD F^\dt$ in $[0,n]$,
%\item $\hga_i = 0$ for each nondegree $i$ of $(\AD)^2 F^\dt$ in $[0,n]$.
%\end{itemize}
%Similarly the equations sets for the $\beta_i$ and $\gamma_i$ are
%equivalent to these.
\end{proof}

\subsection{Uniqueness of  Betti numbers}

Given a balanced degree triplet $\tri = (A,B,C)$. The set 
$A$ is a subset of $[c, n-b]$, containing the end points of this
interval. Let us suppose that there is an internal nondegree of $A$,
i.e. $A$ is a proper subset of $[c, n-b]$. 
Let $A$ contain $[c, c+t-1]$ but not $c+t$. 
The set $B$ is a subset of $[a, n-c]$ containing the endpoints.
Let $s\geq 1$ be maximal such that $\trs{B} \sus [c, n-b]$ 
is disjoint from the interval 
$[c+1, c+s-1]$. Since the degree triangle is balanced we have $s \leq t$.
Let $\tri^\prime = (A^\prime, B^\prime, C)$ where 
\[ A^\prime = A \cup \{ c+t \} \backslash [c, c+s-1], 
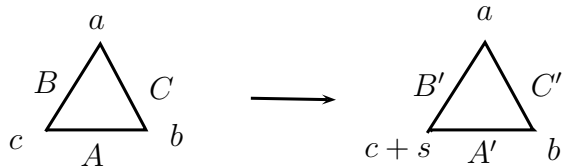
\quad \overline{B^\prime} = \overline{B} \backslash \{ c\}. \]

% Generated with LaTeXDraw 2.0.8
% Wed Sep 12 14:55:40 CEST 2012
% \usepackage[usenames,dvipsnames]{pstricks}
% \usepackage{epsfig}
% \usepackage{pst-grad} % For gradients
% \usepackage{pst-plot} % For axes
\scalebox{1} % Change this value to rescale the drawing.
{\put(90,-15){
\begin{pspicture}(0,-1.163125)(7.7090626,1.163125)
\psline[linewidth=0.04](0.65,-0.6103125)(1.37,0.5496875)(1.95,-0.5503125)(1.97,-0.5903125)(0.67,-0.5903125)
\usefont{T1}{ptm}{m}{n}
\rput(0.25453126,-0.7503125){$c$}
\usefont{T1}{ptm}{m}{n}
\rput(2.3745313,-0.6603125){$b$}
\usefont{T1}{ptm}{m}{n}
\rput(1.3245312,0.8396875){$a$}
\usefont{T1}{ptm}{m}{n}
\rput(1.2645313,-0.9203125){$A$}
\usefont{T1}{ptm}{m}{n}
\rput(0.65453125,0.0396875){$B$}
\usefont{T1}{ptm}{m}{n}
\rput(2.1745312,-0.0403125){$C$}
\usefont{T1}{ptm}{m}{n}
\rput(5.284531,-0.8503125){$c+s$}
\usefont{T1}{ptm}{m}{n}
\rput(7.3345313,-0.8403125){$b$}
\usefont{T1}{ptm}{m}{n}
\rput(6.4245315,0.9596875){$a$}
\psline[linewidth=0.04cm,arrowsize=0.05291667cm 2.0,arrowlength=1.4,arrowinset=0.4]{<-}(4.47,-0.1903125)(3.35,-0.1703125)
\psline[linewidth=0.04](5.67,-0.6303125)(6.43,0.5696875)(7.05,-0.5703125)(7.07,-0.5903125)(5.73,-0.5903125)(5.71,-0.6103125)
\usefont{T1}{ptm}{m}{n}
\rput(6.3945312,-0.8803125){$A^\prime$}
\usefont{T1}{ptm}{m}{n}
\rput(7.244531,0.0196875){$C^\prime$}
\usefont{T1}{ptm}{m}{n}
\rput(5.6945312,0.0196875){$B^\prime$}
\end{pspicture} 
}}

\begin{lemma}
If $\tri$ is a balanced degree triplet, then $\tri^\prime$ is 
a balanced degree triplet.
\end{lemma}

\begin{proof}
If $\tri$ has $e$ internal nondegrees, 
then clearly $\tri^\prime$ has $e-s$ internal nondegrees.
(We remove one nondegree from $A$ and $s-1$ from $\overline{B}$.)
Since $\tri^\prime$ has parameters $c+s, b$ and $a$, the
equation $a+b+c+e = n$ continues to hold when passing from 
$\tri$ to $\tri^\prime$. Viewing $\tri^\prime$ from the corner $c+s$
we see it is balanced here since $\tri$ was. Viewing $\tri^\prime$ from
corner $a$ we see that it is balanced in the interval 
$[v, n-a]$ for $v \geq \max \{ c+s, b\}$ since $\tri$ was, and 
when $v \leq \max \{ c+s, b\}$ we can use the same argument as in 
case 1. in the proof of Proposition \ref{PurecxPropBalance}.
The last case of corner $b$ goes in the same way.
\end{proof}

\begin{proposition}
Let $\tri$ be a balanced triplet, with an internal nondegree on one of its
sides, say side $A$. Then if i) there exists a triplet of pure
free squarefree complexes for the degree triplet $\tri^\prime$, 
and ii) the equation system for $\tri^\prime$ has a one-dimensional
solution, then the equation system for $\tri$ also has a one-dimensional
solution.
\end{proposition}

\begin{proof}
Let $X$ be the coefficient matrix for the system of equations 
given in Proposition \ref{NumProAsystem}
for the Betti numbers 
$\alpha_i$, of a pure complex associated to $A$ in the degree triplet $\tri$. 
Let $X^\prime$ be the corresponding
coefficient matrix for the triplet $\tri^\prime$. By hypothesis
the solution set of $X^\prime$ is one-dimensional. The coordinates of 
a solution vector $(\alpha^\prime_0, \ldots, \alpha^\prime_{r^\prime})$
may be taken as the minors of the matrix $X^\prime$. By hypothesis
there exists a pure free squarefree complex $F^{\prime \dt}$, part of a
triplet,  whose Betti numbers are a multiple of
this solution vector, and hence all the $\alpha^\prime_i$ will be nonzero.

   Now note that the columns in $X^\prime$ has columns parametrizes by
by the degrees of $A^\prime$ in $[c+s, n-b]$. These are exactly the degrees
of $A$ which are in $[c+s, n-b]$, together with the degree $c+t$. 
Write the coefficient matrix $X$ such that equations (\ref{NumLigAto})
are the first rows and the equations (\ref{NumLigAen}) the second
group of rows, and the (\ref{NumLigAtre}).
If we remove the first column in $X$, indexed by $c$, then $X$ will
have a form
\[ \left [ \begin{matrix} T & 0 \\
                          Z & Y
           \end{matrix} \right ]
\]
where $T$ is a triangular matrix of size $(s-1) \times (s-1)$. 
This is due to the hypotheses we have on the forms of $A$ and $\overline{B}$
in the interval $[c,c+s-1]$.
If on the other hand we remove the column of $X^\prime$ indexed by $c+t$, 
we will simply get the matrix $Y$.
Hence the determinant of $Y$, which is one of the $\alpha^\prime_i$,
is nonzero. So the matrix $X$ will have full rank, and hence
a one-dimensional solution set. 
\end{proof}
  
We then get the following.

\begin{theorem} \label{NumTeoremAB}
Part a. of Conjecture \ref{PurecxConj} implies part b. of the conjecture:
The Betti numbers of triplets of pure complexes of free squarefree
modules, associated to a balanced degree triplet, are uniquely
determined up to common scalar multiple.
\end{theorem}

\begin{proof}
This follows from the previous proposition once we know it is true
for the induction start. And the induction start is a degree triplet
with no nondegrees. But in any degree triplet where all nondegrees
are on only one edge, the uniqueness of Betti numbers follows by the 
Herzog-K\"uhl equations, see \cite[Sec. 1.3]{FlIntro}, 
since if this edge corresponds to a complex
$F^\dt$, then this complex is a resolution of a Cohen-Macaulay module,
Lemma \ref{SetLemCM}.
The uniqueness
of all Betti numbers up to common scalar multiple follows by the 
transition equations  (\ref{NumLigAtilB}), (\ref{NumLigBtilC}), and 
(\ref{NumLigCtilA}).
\end{proof}

\section{Construction of triplets when the internal nondegrees
are on only one side of the degree
triangle}

\label{KonSec}
In this section we construct triplets of pure squarefree 
complexes in the case that two of the complexes are linear.
These correspond to degree triangles where two of the sides
only consists of degrees (filled circles). 
%We give two different constructions of this.
%They are derived from the two constructions of pure resolutions of 
%Cohen-Macaulay modules in \cite{EFW}.

\subsection{Auxiliary results on subspaces of vector spaces}

Let $E$ be a vector space and $E_1, \ldots, E_r$ subspaces of $E$.
For $I$ a subset of $[r] = \{1, \ldots, r \}$ we let $E_I$ be the intersection
$\cap_{i \in I} E_i$. 

\begin{lemma} \label{ConLemGen}
Suppose $E_1, \ldots, E_r$ are general subspaces of $E$ of codimension one,
where $r \leq \dim_\kk E$. Then the $E_{[r]\backslash \{i \}}$ as $i$ varies through
$i = 1, \ldots, r$, generate $E$.
\end{lemma}

\begin{proof}
By dividing out by $E_{[r]}$ we may as well assume that $r = \dim_\kk E$.
Then each $E_{[r]\backslash \{i \}}$ corresponds to a one-dimensional vector
space.
To construct the $E_i$ we may chose general vectors $v_1, \ldots, v_r$
and let $E_i$ be spanned by the $(r-1)$-subsets of this $r$-set we get
by successively omitting the $v_i$. 
\end{proof}

\begin{lemma} \label{ConLemCode}
Let $E_i$ be a subspace of $E$ of codimension $e_i$ for $i = 1, \ldots, r$.
Suppose for each {\it proper} subset $J$ of $I$ that 
the codimension of $E_J$ is $\sum_{i \in J} e_i$. 
If $\codim E_I < \sum_{i \in I} e_i$, then the $E_{I \backslash \{i \}}$ do not
generate $E$ as $i$ varies through $I$. 
\end{lemma}

\begin{proof}
Let the codimension of $E_I$ be $(\sum_{i \in I} e_i) - r $ where $r > 0$.
By dividing out by $E_I$ we may assume $E_I = 0$ and so this number is
the dimension of $E$. Then the dimension of
$E_{I \backslash \{i \}}$ is $e_i - r$, and so if $|I| \geq 2$ these
cannot generate the whole space $E$.
\end{proof}

\noindent{\bf Notation.} We shall in the following denote by $S^r(E)$ the
$r$'th symmetric power of $E$ and by $D^r(E)$ the $r$'th divided power
of $E$. Also let $\tilde{D}^r(E) = \wedge^{\dim_\kk E} E \te_\kk D^r(E)$. 

\subsection{Construction of tensor complexes}
\label{ConSubsecTensor}
We start with a degree triplet $(A,B,C)$ where $B = [a, \overline{c}]$ and 
$C = [b, \overline{a}]$ are intervals, 
i.e. contain no nondegrees. We partition the complement of $A$ in
$[0,n]$ into successive intervals
\[ [u_0 + 1, u_0 + w_0 - 1], [u_1 +1, u_1 + w_1 - 1], 
\ldots, [u_r+1, u_r + w_r - 1], [u_{r+1}+1, u_{r+1} + w_{r+1} - 1], \]
where for the first and last interval we have $u_0 = -1, w_0 = c+1$
and $u_{r+1} = n-b$ and $w_{r+1} = b+1$, and for the middle intervals
$c \leq u_1$, $u_i + w_i \leq u_{i+1}$,  and $u_r + w_r \leq n-b$.
Let $W_i$ be a vector space of dimension $w_i$, and $W = \te_{i = 0}^{r+1} W_i$.
Denote by $\ovpil{W}$ the tuple $(W_0, \ldots, W_{r+1})$. 
Let $V$ be a vector space of dimension $n$ 
and $S(V \te W^*)$ the symmetric algebra. 
In the language of \cite[Sec. 5]{BEKS2}, $(0; u_0, \ldots, u_{r+1})$ is
a {\it pinching weight} for $V,\ovpil{W}$. 

  Berkesch et.al. \cite{BEKS2}, construct a resolution $F^\dt(V;\ovpil{W})$
of pure free 
$S(V \te W^*)$-modules
with degree sequence $A$ such that the term with free generators of
degree $d \in A$ has the form:
\begin{equation} \label{ConLigTensor}
 \bigwedge^d V \bigotimes (\te_{d \leq u_i} S^{u_i-d}(W_i)) 
\bigotimes (\te_{d \geq u_i + w_i}
\tilde{D}^{d-u_i - w_i}(W_i)) \bigotimes S(V \te W^*).
\end{equation}
This complex is a resolution of a Cohen-Macaulay module and is equivariant for
the group
\[GL(V) \times GL(W_0) \times \cdots \times
GL(W_{r+1}). \]

The construction of this complex follows the method of Lascoux, presented
in \cite[Sec. 5.1]{We}. 
Let $\pwov$ be the product $\PP(W_0) \times \cdots
\times \PP(W_{r+1})$. There is a tautological sequence:
\[ 0 \pil \gS \pil W \te \gO_{\pwov} \pil
\gO_{\pwov}(1, \ldots, 1) \pil 0. \]
Dualizing this sequence and tensoring with $V$
we get a sequence (let $\gQ = \gS^*$)
\[  0 \pil V \te \gO_{\pwov}(-1, \ldots, -1)  \pil 
V \te W^* \te \gO_{\pwov} \pil
V \te \gQ \pil 0. \]
Constructing the affine bundles over $\pwov$ of the last two terms in this 
complex, we get a diagram
\[
\xymatrix{
Z = \VV(V \te \gQ) \ar@{^{(}->}[r] \ar[d]
 & \VV((V \te W^*) \ar[d]^\pi
\te \gO_{\pwov} ) \ar@{=}[r]  & \VV(V \te W^*) \times \pwov \\
Y \ar@{^{(}->}[r]  & \VV(V \te W^*)  &   }
\]
where $Y$ is the image of $Z$ by the projection $\pi$. 
The projection of the structure sheaf $\pi_*(\gO_Z)$ is the sheaf on 
the affine space $Y$ associated to the $S(V \te W^*)$-module
$H^0(\pwov, \, \Sym(V \te \gQ))$.
%, by \cite[Prop. 5.1.2.b]{We} 
%And the resolution above
%is the resolution of this $S$-module.

Let $p$ be the projection of $\VV(V \te W^*) \times \pwov)$
to the second factor.
Let $\gL$ be the line bundle
$\gO_{\pwov}(u_0, \ldots, u_{r+1})$ on $\pwov$.  
Then $M = H^0(\pwov, \Sym(V \te \gQ) \te p^* \gL)$ is an 
$S(V \te W^*)$-module and the complex $F(V; \ovpil{W})$ is a resolution
of this module, by \cite[Prop. 5.1.2.b]{We}. 
The sheafification of this module on the affine space
is in fact $\pi_*(\gO_Z \te p^* \gL)$. 

\medskip
\noindent{\bf Fact.} $\dim Y = \dim Z$. This is argued for in \cite{BEKS2},
see for instance the proof of Proposition 3.3. First note that
\begin{equation*} \dim Z = \dim \pwov + n \cdot \rk \gQ.
\end{equation*}
Since $F(V;\ovpil{W})$ is a resolution of a module supported on 
$Y$, the length of this resolution is at least the codimension of $Y$.
Hence  
\begin{align*}
 \dim Y & \geq  n \dim_\kk W - |A| + 1 \\
   & = n \dim_\kk W - n + \sum_i (w_i - 1).
\end{align*}
Since $\rk \gQ = \dim_\kk W - 1$ and $\dim \pwov = \sum_i (w_i - 1)$
we get $\dim Y \geq \dim Z$ and we obviously also have the opposite inequality.

\subsection{Degeneracy loci of bundles}
Let $\gE$ be a vector bundle, i.e. a locally free sheaf of finite rank $e$,
on a scheme $S$. Let $T$ be a subspace of the sections $\Gamma(S, \gE)$.
The map $T \te_\kk \gO_S \pil \gE$ defines a map and an exact sequence
\begin{equation} \label{ConLigRkvot}
 T \te_\kk \Sym(\gE) \pil \Sym(\gE) \pil \gR \pil 0.
\end{equation}
where the cokernel $\gR$ is a quasi-coherent sheaf of $\gO_S$-algebras.
The space $T$ gives global sections of the affine bundle 
$\VV = \VV_S(\gE)$ and they generate
a sheaf of ideals of $\gO_\VV$ defining a subscheme $\gX = \Spec_{\gO_S} \gR$.

Now we may stratify $S$ according to the rank of the map 
$T \te_\kk \gO_S \pil \gE$. Let $U_i$ be the open subset where the rank is
$\geq \dim_\kk T-i$. Then if $x \in U_i\backslash U_{i-1}$ we get an
exact sequence 
\[ T \te_\kk \Sym(\gE_{\kk(x)}) \pil \Sym(\gE_{\kk(x)}) \pil \gR_{\kk(x)} \pil 0 \]
where $\gR_{\kk(x)}$ is the quotient symmetric algebra generated by a 
vector space of dimension $e-t + i$. Hence the fiber
$\gX_{\kk(x)}$ has dimension $e-t+i$. We observe that the dimension of $\gX$ is 
less than or equal to the maximum of 

\begin{equation} \label{KonLigMaxdim}
\max \{ \dim (S \backslash U_{i-1}) + e - t + i \}.
\end{equation}

\medskip
We adapt this to the situation of Subsection \ref{ConSubsecTensor}
so $S = \pwov$. 
Let $V$ be a vector space with a basis $x_1, \ldots, x_n$ and 
$\gE = V \te_\kk \gQ$. For each $x_i$ chose a general subspace $E_i \sus W^*$
of codimension one. Let $T$ be the subspace
\[ \bigoplus_i x_i \te E_i \sus \bigoplus_i x_i \te W^*
=  V \te_\kk W^*. \]
Note that the dimension of $T$ equals the rank of $V \te_\kk \gQ$. 

\begin{proposition} \label{KonProDeg}
The locus where the composition
\[ \alpha: T \te_\kk \gO_{\pwov} \inpil V \te_\kk W^* \te_\kk \gO_{\pwov} \pil  
V \te_\kk \gQ
\] 
degenerates to rank
$\dim_\kk T - i$, has codimension $\geq i$. 
\end{proposition}

\begin{proof} The map $\alpha$ is the direct sum of maps
\[ \alpha_i:  E_i \te_\kk \gO_{\pwov}  \pil \gQ. \]
The rank of $\alpha$ is then the sum of the ranks of
these maps.
Now fix a subset $K$ of $[n]$, and let $E_K = \cap_{i \in K} E_i$.
For each $i \in K$ also fix a number $q_i \geq 1$. 
Let $X$ be the locus of points in $\pwov$ where the image of $\alpha_{i,\kk(x)}$
has corank $\geq q_i $ for $i \in K$.
Let $X^\prime$ be the locus of points in $\pwov$ where
 $E_K \te_\kk \gO_{\pwov} \mto{\alpha_K} \gQ$ degenerates to corank
$\geq \sum_{i \in K} q_i$. We will show that i) 
either $X$ is empty, or $X \sus X^\prime$, and 
%$\sum_{i \in K} q_i \leq \rk \gQ$, and 
ii) $\codim X^\prime \geq \sum_{i \in K} q_i$.
This will show the proposition.

\medskip 
i) Suppose $X$ is nonempty, 
and let $x \in \pwov$ be a point where the image of $\alpha_{i, \kk(x)}$ has
corank $ \geq q_i$. 
Clearly for any $I \sus [n]$, the image of $\alpha_{I, \kk(x)}$ is 
contained in $\cap_{i \in I} \im \alpha_{i,\kk(x)}$. 
Suppose there is  $I \sus K$ such that this intersection 
in $\gQ_{\kk(x)}$
does not have corank $\geq \sum_{i \in I} q_i$ and let $I$ be minimal such 
in $K$.
Then clearly $|I| \leq \rk \gQ + 1$ since all $q_i \geq 1$.  
The image of $E_{I \backslash \{j \}}$
is contained in the intersection 
$\cap_{i \in I \backslash \{j \}}  \im \alpha_{i,\kk(x)}$ and these
do not generate $\gQ_{\kk(x)}$ by Lemma \ref{ConLemCode}. 
By Lemma \ref{ConLemGen} this is 
not possible since the $E_{I \backslash \{j \}}$ generate $W^*$,
and the map $W^* \te \gO_{\pwov} \pil \gQ$ is surjective.
Hence the image of $E_K \mto{\alpha_{K, \kk(x)}} \gQ_{\kk(x)}$ must be of corank
$\geq \sum_{i \in K} q_i$. This proves the first part i).

\medskip
ii)  The image
of $E_K \te_\kk \gO_{\pwov} \mto{\alpha_K} \gQ$ is a sheaf of corank
$\geq \rk \gQ - \dim_\kk E_K = |K| -1$ at all points $x$ in $\pwov$. 
Since $\gQ$ is generated
by its global sections, the locus of points
where this map has corank $\geq c + |K| -1$, for some $c \geq 1$,
has codimension 
in $\pwov$ greater than or equal to
\begin{align}
c(c+\rk \gQ - \dim_\kk E_K) & =  c(c+ |K| -1)  \notag\\
\label{ConLigCodim} & \geq c + |K| -1.
\end{align}
Hence the locus of points where the corank is $\geq \sum_{i \in K} q_i 
= c + |K| -1$ has codimension $\geq \sum_{i \in K} q_i$. 
\end{proof}

\begin{corollary} \label{KonCorDim}
Let $T$ be the sections of $\gE = V \te_\kk \gQ$ given by the composition
$\alpha$. 
The subscheme $\gX = \Spec_{\gO_{\pwov}} \gR$ of  
$\VV_{\pwov} (\gE)$ defined by the vanishing of $T$, see (\ref{ConLigRkvot}),
has dimension less than or equal to the dimension of $\pwov$. 
\end{corollary}

\begin{proof}
This follows by the above Proposition \ref{KonProDeg} and
the expression for the dimension given by (\ref{KonLigMaxdim}).
\end{proof}

\subsection{Construction of pure free squarefree resolutions 
from tensor complexes}
Recall that $x_1, \ldots, x_n$ is a basis for $V$.
Consider the map
\begin{equation} \label{ConLigKvot}
 V \te W^* = \bigoplus_{i=1}^n x_i \te W^* \pil 
\bigoplus_{i = 1}^n x_i \te (W^*/E_i) \iso V. 
\end{equation}
This identifies $V$ as the quotient space of $V \te W^*$ by the 
subspace $T$. 
It induces a homomorphism of algebras
\begin{equation*} S(V \te W^*) \pil S(V).
\end{equation*}
Recall the pure resolution $F^\dt(V;\ovpil{W})$ of Subsection 
\ref{ConSubsecTensor} whose degree sequence
is given by the set $A$.

\begin{proposition} \label{KonProEksistens}
The complex 
\[ F^\dt = F^\dt(V;\ovpil{W}) \te_{S(V\te W^*)} S(V) \]
is a pure free squarefree resolution of a Cohen-Macaulay squarefree
$S(V)$-module. Its degree
sequence is $A$.
\end{proposition}

\begin{remark} \label{KonRemTensor}
The essential thing about the tensor complex $F(V;\ovpil{W})$ 
that makes this construction
work is that in the generators of its free modules
in (\ref{ConLigTensor}), the only representations of $V$ that occur
are the exterior forms $\wedge^d V$. Choosing a basis $x_1, \ldots, x_n$
for $V$, this is generated by (squarefree) exterior monomials.
This is why the tensor complexes are ``tailor made'' for our construction.
\end{remark}

\begin{remark}
In our construction we could equally well have used the 
$GL(F) \times GL(G)$-equivariant complex of \cite[Sec.4]{EFW}.
Again in the generators of the free modules, the $\wedge^d F$ are
the only representations of $F$ that occur. In contrast
all kinds of irreducible representations of $G$ are involved, and it also only
works when char.$ \,\kk = 0$, which is why we focus on the 
tensor complexes of \cite{BEKS2}.
\end{remark}

From this we obtain as a corollary the following.

\begin{theorem} \label{KonTeoremHoved}
For a balanced degree triplet $(A,B,C)$ of type $n$, where $B$ and $C$
are intervals, i.e. the only internal nondegrees are in the 
interval associated to $A$, 
there exists a triplet of pure free squarefree modules over the
polynomial ring in $n$ variables whose degree sequences are given by 
$A,B$ and $C$. 
\end{theorem}

\begin{proof}[Proof of Theorem \ref{KonTeoremHoved}]
Let the endpoints of $A$ be $c$ and $n-b$. All nondegrees of this triplet
are in $[c,n-b]$, so the number $e$ of such is the cardinality of 
$[c, n-b]\backslash A$. Since $a+b+c + e = n$ we see that $a$ is 
determined by $A$. Hence $B$ and $C$ are determined by $A$. 
Starting with the complex $F^\dt$ in Proposition \ref{KonProEksistens}, 
by Lemma \ref{SetLemCM}
$(\AD)(F^\dt)$ and $(\AD)^2(F^\dt)$ are both linear and so have 
degrees given by $B$ and $C$. 
\end{proof}

\begin{remark}
In the forthcoming paper \cite{FlZip} we consider the construction
of triplets of pure complexes in general. We transfer Conjecture
\ref{PurecxConj} to a conjecture on the existence of certain complexes
of coherent sheaves on projective spaces. In the case of the above theorem,
these complexes reduce to a single coherent sheaf, the line bundle
$\gO_{\pwov}(u_0, \ldots, u_{r+1})$ on the Segre embedding of $\pwov$ 
in the projective space ${\mathbb P}(W)$. 
\end{remark}

\begin{proof}[Proof of Proposition \ref{KonProEksistens}]
Let $Z^\prime$ be the pullback in the diagram
\[ 
\xymatrix{ Z^\prime \ar@{^{(}->}[r] \ar@{^{(}->}[d] &\VV(V) \times \pwov 
\ar@{^{(}->}[d]\\
 Z \ar@{^{(}->}[r] &\VV(V \te W^*) \times \pwov.
}
\]
The subscheme $\VV(V)$ of $\VV(V \te W^*)$ is defined by the vanishing
of the subspace $T$ of $V \te W^*$.
Since $Z = \VV_{\pwov}(V \te_\kk \gQ)$ we see that $Z^\prime$ is the 
subscheme of $Z$ defined by the vanishing of the sections $T$ of 
$V \te \gQ$ given by the composition $\alpha$ in Proposition \ref{KonProDeg}.
By Corollary \ref{KonCorDim} the dimension of $Z^\prime$ is less than or equal
to $\dim \pwov$. 
Since $\dim_k T$ equals the rank of $V \te \gQ$, the dimension of $Z$ is 
$\dim \pwov + \dim_k T$ 
and so $\dim Z^\prime \leq \dim Z - \dim_\kk T$. 

Let $Y^\prime$ be the pullback in the diagram
\[
\xymatrix{Y^\prime \ar@{^{(}->}[r] \ar@{^{(}->}[d]
& \VV(V) \ar@{^{(}->}[d] \\
 Y \ar@{^{(}->}[r] & \VV(V \te W^*).
} \]
Since the image of $Z$ is $Y$, the image of $Z^\prime$ is $Y^\prime$. 
This gives 
\[ \dim Y^\prime \leq \dim Z^\prime \leq \dim Z - \dim_k T = \dim Y - \dim_k T. \]

%Now $Y$ is the image of $Z$, so the image of $Z^\prime$ will then be 
%$Y^\prime$.
%Since $Y$ and $Z$ have the same dimension we get 
%\[ \dim Y  - \dim Y^\prime \geq \dim Z - \dim Z^\prime. \]
%Now $Z$ has dimension $\dim \pwov + \dim_k \cdot \rk \gQ$ which is 
%$\dim \pwov + \dim_k T$. Now by Proposition \ref{KonProDim} the dimension 
%of $Z^\prime$ is $\dim \pwov$, and so the dimension of $Y^\prime$ is 
%$\leq \dim Y - \dim T$.  
The complex $F^\dt(V;\ovpil{W})$ is a resolution of a Cohen-Macaulay module
$M$ supported on $Y$.
The module  $M^\prime = M \te_{S(V \te W^*)} S(V)$ where
$S(V) = S(V \te W^*)/(T)$ is supported on $Y^\prime$ and so
\[ \dim M^\prime \leq \dim Y^\prime \leq \dim Y - \dim_k T = \dim M - \dim_k T.\]
Since $M^\prime = M /(T \cdot M)$,
a basis for $T$ must form a regular sequence, and so $M^\prime$ is 
a Cohen-Macaulay module with resolution given by $F^\dt$. 
Therefore $F^\dt$ becomes a pure resolution of a Cohen-Macaulay
where the terms with generators of degree $d \in A$ are 
\begin{equation} \label{KonLigFd}
\bigwedge^d V \bigotimes (\te_{d \leq u_i} S^{u_i-d}(W_i)) 
\bigotimes (\te_{d \geq u_i + w_i}
\tilde{D}^{d- u_i - w_i}(W_i)) \bigotimes S(V).
\end{equation}

%Therefore the difference between the dimensions of $Y$ and $Y^\prime$ 
%is greater or equal to 
%the difference of dimensions between $Z$ and $Z^\prime$. 
%Since we $Y^\prime$ is obtained from $Y$ by dividing out by a sequence
%forming a basis for $\oplus e_i \te W^*_i$, this sequence is regular.
%But then this sequence is also regular for the $\Sym(V \te W^*)$-module
%$H^0(\pwov, \pi_*(\gO_Z \te p^* \gL))$  whose resolution is
%given by the complex **.
%Then by tensoring this resolution with $-\te_{\Sym(V \te W^*)} \Sym(V)$
%we get a pure free resolution of a Cohen-Macaulay $\Sym(V)$-module whose term
%with generators of degree $d$ is
%\[ \wedge^d V \te (\te (\te_{d \geq u_i} S^{d-u_i}(W_i) 
%\te (\te_{d \leq u_i + w_i}
%\tilde{D}^{u_i + b_i - d}(W_i) \te \Sym(V).\]
The basis $x_1, \ldots, x_n$ of $V$ induces a maximal torus $D$ of 
$GL(V)$, the diagonal matrices.
The quotient map (\ref{ConLigKvot}) is equivariant for
the torus action where $t = (t_1, \ldots, t_n) \in D$ acts on 
$w = \sum x_i \te w^*_i$ as $t.w = \sum (t_i.x_i) \te w^*_i$. 
Thus the complex above is equivariant for the torus action and so
is $\hele^n$-graded. The action on the term (\ref{KonLigFd}) in the complex
is given by the natural 
actions on $\wedge^d V$ and $S(V)$ and the trivial action on the rest of the
tensor factors. Hence the multidegrees of the generators of the terms
above are of squarefree degree, and
so the resolution is squarefree.
\end{proof}

%\bibliography{/Home/siv7/nmagf/Desktop/Bibliography}

\bibliographystyle{amsplain}
\bibliography{/Home/siv7/nmagf/Desktop/Bibliography}

\end{document}